\documentclass[12pt, psamsfonts ]  {amsart}

\usepackage{amssymb,amsfonts,amsmath,graphicx,lineno}
\usepackage{mathrsfs}
\usepackage{anysize}
\usepackage[colorlinks, citecolor=black, urlcolor=black, linkcolor=black]  {hyperref}
\usepackage{url}
\usepackage{color}
\usepackage[all]  {xy}
\usepackage{tikz}
\usetikzlibrary{shapes, arrows}
\usepackage{pgfplots}
\tikzset{>= angle 60}

\usepackage{graphics}
\usepackage{amsmath}
\usepackage{amsthm}
\usepackage{amssymb}
\usepackage{amscd}
\usepackage{amsfonts}
\usepackage{amsbsy}
\usepackage{epsfig}
\usepackage{color}

\newtheorem{theorem}{Theorem}
\newtheorem{proposition}[theorem] {Proposition}
\newtheorem{definition}{Definition}
\newtheorem*{corollary}  {Corollary}
\newtheorem{lemma}  {Lemma}[section]

\def\R{{\mathbb R}}

\def\Z{{\mathbb Z}}
\def\O{{\mathcal O}}

\def\o{\mathrm o}

\def\lf{\lfloor}
\def\rf{\rfloor}
\def\lc{\lceil}
\def\rc{\rceil}
\begin{document}

\title{Dynamics of 2-interval piecewise affine maps and Hecke-Mahler series}

\maketitle

\centerline{ Michel Laurent and  Arnaldo Nogueira}

\footnote{\footnote \rm 2010 {\it Mathematics Subject Classification:}   
11J91,   37E05. }


\marginsize{2.5cm}{2.5cm}{1cm}{2cm}
  \begin{abstract}
  Let $f : [0,1)\rightarrow [0,1)$ be a $2$-interval piecewise affine increasing  map which is injective but not surjective  (see Figure 1).  Such a map $f$ has a rotation number and can be parametrized by three real numbers. We make fully explicit the dynamics of $f$ thanks to two specific functions  $\delta$ and $\phi$  depending on these parameters  whose definitions involve Hecke-Mahler series. As an application, we show that the rotation number of $f$ is  rational, whenever the three parameters are all algebraic numbers, extending thus the main result of \cite{LN} dealing with the particular case of $2$-interval piecewise affine contractions with constant slope.
  \end{abstract}

\section{Introduction}\label{sec:intro}


\begin{definition} 
Let $I=[0,1)$ be the unit interval. Let $\lambda, \mu, \delta$ be three real numbers. Assume 
 $$
 0< \lambda  <1, \, \mu >0,  \, 1-\lambda < \delta <d_{\lambda,\mu} : = \begin{cases}    1 \quad  if \quad  \lambda \mu <1, 
\\
 {\mu - \lambda\mu  \over \mu-1},     \quad   if  \quad \lambda \mu \ge 1.
\end{cases} 
 $$
 Set $\eta = {1-\delta\over \lambda}$ and  define  a map $f=f_{\lambda,\mu, \delta}: I \to I$  by the splited formula
$$
f(x) =  \begin{cases} \lambda x + \delta,    \quad  if \quad 0\le x < \eta , 
\\
\mu(\lambda x + \delta -1),     \quad   if  \quad \eta \le x <1.
\end{cases}
$$
\end{definition}

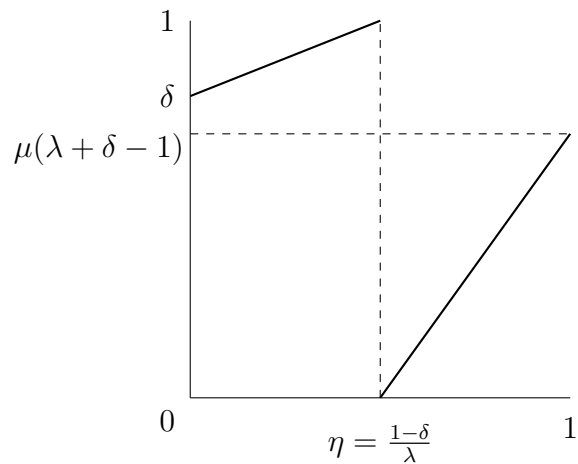
\begin{figure}[ht] 
\begin{tikzpicture}
\draw (0,0) -- (5,0);
\draw (0,0) -- (0,5);
\draw[thick]   (0,4) -- (2.5,5);
\draw[thick]  (2.5,0) -- (5,3.5);
\node   at (-0.3,-0.3) {$0$};
\node at (2.5,-0.6) {$\eta = \frac{1-\delta}{\lambda}$};
\node   at (-0.3,5) {$1$};
\node   at (-0.3,4) {$\delta$};
\draw[dashed]   (0,3.5) -- (5,3.5);
\node   at (-1.2,3.3) {$\mu(\lambda + \delta -1)$};
\node   at (5,-0.4) {$1$};
\draw[dashed]  (2.5,0) -- (2.5,5);
\end{tikzpicture}
\caption{ A plot of $f_{\lambda, \mu, \delta}$}
\end{figure}

The restrictions of $f$ to the  intervals $[0,\eta)$ and $[\eta,1)$  are  increasing affine functions with slopes $\lambda$ and $\lambda\mu$,  respectively.  The symmetry $(x,y)\mapsto (1-x,1-y)$ on $[0,1]^2$  exchanges the  two slopes. Our assumption that the first segment of the graph has  a slope $\lambda$  less than $1$ is thus unrestrictive. Observe that $1-\lambda < d_{\lambda,\mu} \le 1$ and that the bound $\delta \le d_{\lambda,\mu}$ yields the injectivity of $f $.  Indeed the inequality
  $$
\delta > \mu(\lambda + \delta -1)
  $$
   holds true for any $\delta$ in the interval $1-\lambda < \delta < 1$ when $\lambda \mu <1$ ($f$ is then a piecewise contracting map) and is equivalent to 
   $$
   \delta <  {\mu-\lambda\mu\over \mu-1} = d_{\lambda,\mu},
   $$
    when $ \lambda \mu \ge 1$ (see Figure 1). Notice  that in
   the limit case $\delta =d_{\lambda,\mu}$ with $\lambda \mu >1$,  $f=f_{\lambda,\mu, \delta}$ becomes a bijection which was studied in  \cite{Bo}.

We are concerned with the dynamics of the   family of interval maps $f=f_{\lambda,\mu, \delta}$.
We plan to relate their dynamics to the so-called Hecke-Mahler series in two variables $\lambda$ and $\mu$  (see Section 2 for definitions).
The paper \cite{LN} deals with the case $\mu =1$ where the slope is constant. Although the map $f=f_{\lambda,\mu, \delta}$ is not necessarily a piecewise contraction, we extend part of the results established in \cite{LN} to a $2$-slope setting.

 The dynamics of interval piecewise affine contractions  has been studied by many authors, amongst others \cite{Br, Bu, BuC, Co, DH, FC, Ha, NS,NP,NPR}. 
According to   \cite{RT},  every   map $f=f_{\lambda,\mu, \delta}$ has a rotation number  $\rho = \rho_{\lambda,\mu, \delta}$, $0<\rho<1$. Although $f$ is not necessarily a piecewise contraction,
 we will prove that if $ \rho $ takes an irrational value,  then the closure $\overline{C}$ of  the limit set  $C:= \cap_{k\ge 1}f^k(I)$ of $f$  is a Cantor set  and $f$ is topologically conjugated to the rotation map $x \in I \mapsto x + \rho   \mod 1$ on C. When the  rotation number is rational,  the map $f$ has at most one periodic orbit (exactly one in most cases)  and the limit set $C$ equals the periodic orbit when it does exist. More precisely, either $f$ or $f^-$ (a slight modification of the map $f$ whose definition is postponed to Section 8) has a periodic cycle.

 We make the above mentioned qualitative results fully explicit thanks to  formulae involving Hecke-Mahler series. 
Our approach is based on the study of a conjugation function $\phi$ which may be  written down in terms of Hecke-Mahler series. The method was already performed in the special case $\mu=1$ where the two slopes are equal. It is motivated by Coutinho's thesis \cite{Co} and has been recently reworked in \cite{BoSa,JaOb,LN}. The general case involving two different slopes  is quite similar.  Theorem 1 gives  the value of the rotation number $\rho_{\lambda,\mu,\delta}$ in terms of the values of the three parameters $\lambda,\mu,\delta$, while Theorem 3 describes the behaviour of  the orbits of $f$ and their relations with the conjugation $\phi$.

 Next we introduce some standard notations. 
 For any real function $f(x)$ of the real variable $x$, we denote by 
 $$
 f(x^-) = \lim_{y\nearrow x} f(y) \quad {\rm and}\quad  f(x^+) = \lim_{y\searrow x} f(y),
 $$
respectively, the left limit and the right limit  of $f$ at $x$, whenever these limits do exist. As usual $\lf x \rf$ and $\lc x \rc$ stand, respectively, for the {\it integer floor} and the {\it integer ceiling} of the real number $x$. In particular, we have $\lc x \rc = \lf x \rf +1$ for any real number $x\notin \Z$  and $\lc x \rc = \lf x \rf $ when  $x\in \Z$. We denote by $\{ x \}= x -\lf x \rf $ the {\it fractional part} of $x$. The length of an interval $J \subset \R$ is denoted by $| J |$.

We first define a real function $\delta(\lambda ,\mu ,\rho)$ as follows. 

\begin{definition} For positive real numbers  $\lambda, \mu$ and $\rho$ such that $\lambda\mu^\rho < 1$, set
$$
\sigma=  \sigma(\lambda,\mu,\rho):=\sum_{k\ge 1}\left(  \left\lfloor (k+1) \rho \right\rfloor   -  \left\lfloor k \rho \right\rfloor   \right)\lambda^{k} \mu^{\lfloor k \rho\rfloor  }
$$
and
$$
\delta(\lambda,\mu, \rho)={(1-\lambda) (1+ \mu \sigma)\over 1+(\mu-1)\sigma}.
$$
For real numbers $\lambda$ and $\mu$ with $0< \lambda <1$ and $\mu >0$, set
$$
r_{\lambda,\mu} =
 \begin{cases}    1 \quad  if \quad  \lambda \mu <1, 
\\
 {-\log\lambda  \over \log\mu},     \quad   if  \quad \lambda \mu \ge 1.
\end{cases} 
$$

\end{definition}

\begin{figure}[ht]
\begin{tikzpicture}
\node at (0,5)
{\includegraphics[width=.33\textwidth]{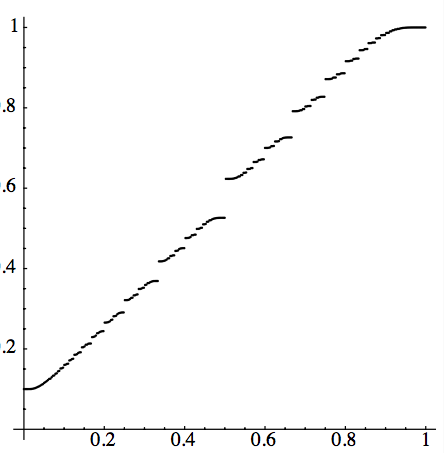}} ;
\end{tikzpicture}
\caption{Plot of the map $\rho \mapsto \delta(0.9,0.8, \rho)$ }
\end{figure}

The series $\sigma(\lambda,\mu,\rho)$ converges when $0 \le \rho < r_{\lambda,\mu}$. 
 For fixed $\lambda$ and $\mu$ with $0 < \lambda<1, \mu >0 $, the map $\rho \mapsto \delta(\lambda,\mu,\rho)$ is increasing in the interval $0\le \rho <r_{\lambda,\mu}$ and it has a left discontinuity at each rational value (see Figure 2). It is continuous for any irrational $\rho$  and right continuous everywhere. 
 The function $\delta$ enables us to compute the rotation number of $f_{\lambda,\mu,\delta}$ thanks to the

\begin{theorem} 
Let $\lambda$ and $\mu$ be real numbers with $0<\lambda<1$ and $\mu >0$. Then the application   $\delta \mapsto \rho_{\lambda,\mu,  \delta}$
is a continuous non decreasing function sending the interval $(1-\lambda, d_{\lambda,\mu})$ onto the interval $(0, r_{\lambda,\mu})$ and satisfying the following properties: 
\\
(i) Let $p/q$ be a rational number with $ 0 < p/q < r_{\lambda,\mu}$  where $p$ and $q$  are relatively prime integers. Then   $  \rho_{\lambda,\mu,\delta}$ takes the value $p/q $ if, and only if, $\delta$ is located in the interval
$$
 \delta\left(\lambda,\mu,(p/ q)^-\right) \le \delta \le \delta\left(\lambda,\mu , p/ q\right)
$$
with the explicit formulae
$$
\begin{aligned}
\delta\left(\lambda,\mu , p/ q\right) &= { (1-\lambda)(1 +\mu S + \lambda^{q-1}\mu^p(1-\lambda)) \over 1 +(\mu-1)S + \lambda^{q-1}\mu^{p-1}(\mu-\lambda\mu-1) }
\\
\delta\left(\lambda,\mu,(p/ q)^-\right) & = 
{(1-\lambda) (1+ \mu S)\over 1 +(\mu -1)S - \lambda^q\mu^{p-1}}
\end{aligned}
$$
where
$$
S  = S(\lambda, \mu, p/q) := \sum_{k=1}^{q-2}  \left( \left\lfloor (k+1) p/q \right\rfloor   -  \left\lfloor k p/q \right\rfloor   \right)\lambda^{k}\mu^{\lfloor kp/q\rfloor  }
$$
and the  sum $S$ equals $0$ when $q=2$.
\\
(ii) For every irrational number $\rho$ with $0 < \rho < r_{\lambda,\mu}$, there exists one and only one 
real number $\delta$ such that $1-\lambda <\delta < d_{\lambda,\mu}$ and   $\rho_{\lambda,\mu, \delta}= \rho$
 which is given by 
$$
\delta=\delta(\lambda,\mu, \rho).
$$
\end{theorem}

 Roughly speaking, the two maps $\rho \mapsto \delta(\lambda,\mu,\rho)$ and $\delta \mapsto \rho_{\lambda,\mu,\delta}$ are ``inverse'' from each other, meaning that their  graphs are symmetric with respect to the main diagonal.  In  the special case $\mu =1$, we recover the formulae obtained in \cite{LN} for the map $f_{\lambda,1,\delta}$, which coincides with the contracted rotation $x\mapsto \{\lambda x + \delta\}$.  
 Notice that the formulae of our Theorem 1 are consistent with those of Theorem 4.15 in \cite{Br},  dealing with the subfamily of contractions $f_{\lambda, \mu,\delta}$ with $\lambda\mu <1$, although the formulations greatly differ.

Applying now a classical transcendence result,  which is stated as  Theorem 4 below,   to the number $\delta(\lambda,\mu,\rho)$,
  we  deduce from the assertion (ii) of Theorem 1 the following:

\begin{theorem} 
Let $\lambda, \mu, \delta$ be algebraic real numbers with $0<\lambda <1$ $ \mu >0$ and   $1-\lambda <\delta<d_{\lambda,\mu}$. Then,  the rotation number $\rho_{\lambda,\mu, \delta}$ takes a rational value.
\end{theorem}

Notice that  Theorem 2 no longer holds for  the value $\delta= d_{\lambda,\mu}$  when $\lambda\mu >1$.  Indeed, $d_{\lambda,\mu} = { \mu -\lambda\mu\over \mu -1}$ is an algebraic number when $\lambda$ and $\mu$ are algebraic, while $f$ has rotation number $-\log\lambda/\log\mu$ by \cite{Bo}. This ratio
is a transcendental number when $\lambda$ and $\mu $ are non-zero algebraic numbers, unless $\lambda$ and $\mu$ are multiplicatively dependent.

We now investigate the behaviour of the iterates of $f= f_{\lambda,\mu,\delta}$ thanks to an explicit conjugation map $\phi$.

\begin{definition} Let $\lambda,\mu,\rho$ be three positive real numbers such that $\lambda\mu^\rho <1$, and let $\delta$ be an arbitrary real number.  Let 
 $\phi_{\lambda,\mu,\delta, \rho}  : \R \to \R$ be the real function defined   by the convergent series
$$
\phi_{\lambda,\mu,\delta, \rho}(y) =  \lfloor y\rfloor   + {1-\delta\over \lambda} + \sum_{k\ge 0}\lambda^k\mu^{\lfloor  y\rfloor   -\lfloor y-k\rho\rfloor  }\left( {\lambda +\delta -1 \over \lambda} + \lfloor y-(k+1)\rho \rfloor   - \lfloor y-k\rho\rfloor  \right).
$$
\end{definition}

\begin{figure}[ht]
\begin{tikzpicture}
\node at (0,5)
{\includegraphics[width=.33\textwidth]{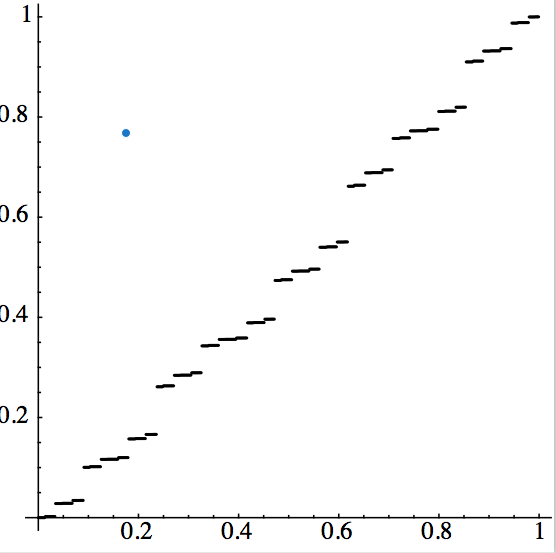}} ;
\end{tikzpicture}
\caption{Plot of the function $\phi_{0.95, 0.9, \delta, (\sqrt{5}-1)/2}$ in the range $0\le y \le1$, where $\delta=\delta(0.95,0.9,(\sqrt{5}-1)/2)=0.6617...$. }
\end{figure}


\eject

\begin{theorem}
Let $\lambda, \mu, \delta$ be three real numbers with
 $0< \lambda  <1$, $\mu >0 $ and  $1-\lambda < \delta <d_{\lambda,\mu}$. 
Set $ \rho= \rho_{\lambda,\mu,\delta}$
 and $\phi=\phi_{\lambda,\mu,\delta,\rho}$.
 \\
(i) Assume that $\rho$ is  irrational.   Then $C = \phi(I)$ and the restriction of $f=f_{\lambda,\mu,\delta}$ to the invariant set $C$ is conjugate by $\phi$ to the rotation $R_\rho : y \mapsto y+ \rho \, \hbox{ mod } \, 1$. In other words, we have the commutative diagramm: 
\begin{equation}
\xymatrix{
I \ar[r]^{R_\rho} \ar[d]  _\phi & I \ar[d]^\phi
\\
C\ar[r]^f & C.
}
\label{diag1}
\end{equation}
Moreover  $\overline{\phi(I)}$ is a Cantor set  and for every $x \in I$, the $\omega$-limit set $$
\omega(x):= \bigcap_{n\rightarrow +\infty}\overline{\bigcup_{k\ge n}f^k(x)}
$$
equals  $ \overline{C}$.
\\
(ii)
Assume that  $\rho ={p/ q}$ is rational, where $p$ and $q$ are relatively prime,  and that 
\begin{equation}
 \delta\left(\lambda,\mu,(p/ q)^-\right) \le \delta < \delta\left(\lambda,\mu , p/q\right).
 \label{eq12}
\end{equation}
Then 
$$
C = \phi(I) =\left\{ \phi({m/ q})\,  ; 0 \le m \le q-1\right\}
$$ 
is a cycle of order $q$ and we have the commutative diagramm: 
\begin{equation}
\xymatrix{
\left\{ {m\over q}\,  ; 0 \le m \le q-1\right\}
 \ar[r]^{R_{p/q}} \ar[d]  _\phi & \left\{ {m\over q}\,  ; 0 \le m \le q-1\right\}
 \ar[d]^\phi
\\
C\ar[r]^f & C
}
\label{diag2}
\end{equation}
where $R_{p/q}$ denotes the rotation $ y \mapsto y+ {p\over q} \, \hbox{ mod } \, 1$. 
Moreover, for every $x \in I$, the $\omega$-limit set $\omega(x)$ equals  $ C$.
\\
(iii)
When $\delta = \delta\left(\lambda,\mu , {p\over q}\right)$, the limit set 
 $C$ is empty and $   \phi(I)$ is a finite set with $q$ elements containing $1$. For every $x \in I$, the $\omega$-limit set $\omega(x)$
 coincides with  $ \phi(I)$.
\end{theorem}

The paper is organized as follows. In Section 2, we introduce Hecke-Mahler series and relate them to our functions $\delta$ and $\phi$. Then,  Theorem 2 easily follows from Theorem 1.
The purpose of Sections 3 and 4 is to establish the basic conjugation  equations \eqref{diag1} and \eqref{diag2}. This goal is achieved thanks to Lemma 4.2 where    some relations  connecting the parameter $\delta$ with values of the function $\delta(\lambda,\mu,\rho)$ are needed, as for instance the inequalities \eqref{eq12} in the case \eqref{diag2}. 
It turns out that these constraints characterize the rotation number $\rho_{\lambda,\mu,\delta}$. As a consequence of the method, we establish Theorem 1 in Section 5. The next two sections provide additional information on the dynamics of $f$ in the case of an irrational rotation number (Proposition 5 in Section 6), or a rational one (Proposition 6 of Section 7). 
In both cases, we explicitly describe the  iterated images  $f^n(I), \, n\ge 1$.  Finally Section 8 deals with  the exceptional values of the form  $\delta = \delta(\lambda,\mu,p/q)$ for which no periodic cycle exists.

\eject

 \section{Hecke-Mahler series and transcendental  numbers}
 
 \subsection{On Hecke-Mahler series.}
 We  introduce the following sums:
 \begin{definition}
Let  $\lambda$, $\mu$ and $\rho$ be positive real numbers such that $0< \lambda <1, 0< \lambda  \mu^\rho <1$.   We set,  for every real number $x$,
$$
\begin{aligned}
\Psi_\rho(\lambda,\mu) =& \sum_{k\ge 1}\sum_{1\le h \le k\rho}\lambda^k\mu^h,
\\
\Phi_\rho(\lambda, \mu, x) = & \sum_{k\ge 0}\sum_{0 \le l < k \rho + x}\lambda^k\mu^l,
\end{aligned}
$$
with the convention that a sum indexed by an empty set equals zero.
\end{definition}

Notice that $\Psi_\rho(\lambda,\mu)$ is a right continuous function in the variable $\rho$,
while the function $\Phi_\rho(\lambda, \mu, x)$ is left continuous in both variables $\rho$ and $x$.
 Viewed as  power series in the two variables $\lambda$ and $\mu$, these two functions 
 are called Hecke-Mahler series which have been studied  especially from a diophantine point of view \cite{AdDa, Boh, BoBo, Da, Ko, LoVdPA, LoVdPB, Ni, NST}. We relate our functions $\delta(\lambda,\mu,\rho)$ and $\phi_{\lambda,\mu,\delta,\rho}(x)$ respectively to  $\Psi_\rho(\lambda,\mu)$ and $\Phi_\rho(\lambda, \mu, x)$. 
\begin{lemma}
Let $\lambda,\mu, \rho$ be real numbers with $0 < \lambda<1, \mu >0$ and $ 0 < \rho < r_{\lambda,\mu}$, then 
the following equality holds 
$$
\sigma(\lambda,\mu,\rho) = \sum_{k\ge 1}\left(\lfloor (k+1) \rho \rfloor   -  \lfloor k \rho \rfloor \right) \lambda^k \mu^{\lfloor k\rho\rfloor  }
={ 1-\lambda \over \lambda\mu} \Psi_\rho(\lambda,\mu) .
$$
\end{lemma}

\begin{proof}
Reverting the summation order for the indices $h,k$ involved in $\Psi_\rho(\lambda,\mu)$, we obtain
$$
\Psi_\rho(\lambda,\mu)= \sum_{h\ge 1} \sum_{k\ge {h\over \rho}}\lambda^k\mu^h
= \sum_{h\ge 1} \sum_{k\ge {\left\lc h\over \rho\right\rc }}\lambda^k\mu^h
= {1\over 1-\lambda} \sum_{h\ge 1} \lambda^{\left\lc{h\over \rho}\right\rc }\mu^h .
$$
A positive  integer $k$ is of the form 
$\left\lc{h\over \rho}\right\rc$ for some some positive integer $h$ if and  only if 
$k-1< {h\over \rho} \le k$, or equivalently 
$ (k-1) \rho < h \le k\rho$. There exists at most one integer $h$ in the interval $((k-1)\rho, k\rho]$ whose length is $\rho<r_{\lambda,\mu} <1$. The integer $h$ does exist exactly when $\lf (k-1)\rho\rf  = \lf k\rho\rf -1$ and then $h =\lf k\rho\rf  = \lf (k-1)\rho\rf +1$. Otherwise, $\lf (k-1)\rho\rf   = \lf k\rho\rf  $.
Thus
$$
\begin{aligned}
\sum_{h\ge 1} \lambda^{\left\lc{h\over \rho}\right\rc  }\mu^h =&  \sum_{k\ge 1} 
\left(\lf k\rho \rf   -  \lf (k-1) \rho \rf \right) \lambda^k\mu^{\lf k \rho\rf  } = \lambda \mu \sum_{k\ge 1}\left(\lf k \rho \rf   -  \lf (k -1)\rho \rf \right) \lambda^{k-1}\mu^{\lf (k-1)\rho\rf  } 
\\
=&  \lambda\mu\sum_{k\ge 0}\left(\lf (k +1)\rho \rf   -  \lf k\rho \rf \right) \lambda^k\mu^{\lf k \rho\rf  }
= \lambda\mu\sum_{k\ge 1}\left(\lf (k +1)\rho \rf   -  \lf k\rho \rf \right) \lambda^k\mu^{\lf k \rho\rf  },
\end{aligned}
$$
since $0< \rho <1$.
\end{proof}

We deduce 
\begin{corollary}
Let $0 < \lambda <1 $ and $\mu>0$, then the map $\rho\mapsto \delta(\lambda,\mu,\rho)$ is increasing on the interval $0< \rho<r_{\lambda,\mu}$ and
sends the interval $(0,r_{\lambda,\mu})$ into the interval  $(1-\lambda,d_{\lambda,\mu})$. Moreover,  it is
 right continuous everywhere and continuous at any  irrational point $\rho$. 
 \end{corollary}

\begin{proof} 
Using Lemma 2.1, we can rewrite $\delta(\lambda, \mu, \rho)$ in the form 
\begin{equation}
\delta(\lambda,\mu, \rho) 
= {(1-\lambda) (1+ \mu \sigma(\lambda,\mu,\rho))\over 1+(\mu-1)\sigma(\lambda,\mu,\rho)}
= \mu (1-\lambda) { \lambda + (1-\lambda)\Psi_\rho(\lambda, \mu)\over \lambda \mu + (1-\lambda)(\mu-1)\Psi_\rho(\lambda,\mu)}.
\label{eqdelta}
\end{equation}
We distinguish two cases whether $\lambda\mu <1$ or not.

When $\lambda \mu <1$, we have $d_{\lambda,\mu} = r_{\lambda,\mu}=1$.  The series $\Psi_\rho(\lambda,\mu)$  converges for any $\rho\in (0,1)$ and  
the map $\rho\mapsto \Psi_\rho(\lambda,\mu)$ is obviously  increasing, since $\Psi_\rho(\lambda,\mu)$ is a sum of  powers of $\lambda$ and $\mu$  and that the set of  summation indices $(h,k)$  enlarges when $\rho$ grows. We easily compute that
$$
\lim_{\rho\searrow 0}\Psi_\rho(\lambda, \mu) = 0 \quad {\rm and \, \, that} \quad
\lim_{\rho\nearrow 1}\Psi_\rho(\lambda, \mu) = {\lambda^2\mu\over (1-\lambda)(1-\lambda\mu)}.
$$
It follows that 
$$
0< \Psi_\rho(\lambda,\mu) < {\lambda^2\mu\over (1-\lambda)(1-\lambda\mu)}
 $$  
  for any $\rho\in (0,1)$.
 Since $\lambda$ differs from $0$ and $1$, the homographic function  
 $$
x\mapsto \mu(1-\lambda) { \lambda + (1-\lambda) x\over \lambda \mu + (1-\lambda)(\mu-1)x}
$$
is increasing  on the interval $0 < x <{\lambda^2\mu\over (1-\lambda)(1-\lambda\mu)}$,  
and sends this interval onto $(1-\lambda, 1)$. By composition, we obtain that the  image of  $(0,1)$ by the map $\rho\mapsto \delta(\lambda,\mu,\rho)$ is contained  in the interval $(1-\lambda,1)$.

When $\lambda \mu \ge 1$, we have $ r_{\lambda,\mu}= -\log\lambda/ \log\mu $ and the series  $\Psi_\rho(\lambda,\mu)$  tends to $+ \infty$ when $\rho $ tends to $r_{\lambda,\mu}$ from below. Thus, we obtain in this case,
$$
0< \Psi_\rho(\lambda,\mu) < + \infty
 \quad {\rm and} \quad 
 1- \lambda <   \delta(\lambda,\mu,\rho) < { \mu(1-\lambda)\over \mu -1} = d_{\lambda,\mu}.
 $$

For the continuity's property, observe that the floor function $x\mapsto \lf x\rf$ is right continuous on $\R$ and continuous on $\R \setminus\Z$. 
 \end{proof}

We now  give an alternative formula for the function $\phi_{\lambda,\mu,\delta, \rho}$
in terms of the Hecke-Mahler series $\Phi_\rho$. 

\begin{lemma}
Let $0< \lambda <1$, $\mu >0$ and $0< \rho <r_{\lambda,\mu}$, 
then, for any real number $x > -1$, we have the equalities
\begin{eqnarray}
 \sum_{k\ge 0} \lambda^k\mu^{\lc k \rho + x\rc}  &=&  { 1 \over 1 - \lambda} - (1-\mu)\Phi_\rho(\lambda,\mu , x),
 \label{eq1}
\\
 \sum_{k \ge 0 } ( \lc (k+1)\rho + x \rc - \lc k\rho + x \rc ) \lambda^k \mu^{\lc k\rho + x \rc}  &=&
-{ 1 - \mu^{ \lc x \rc  } \over (1-\mu)\lambda} + {1-\lambda \over \lambda} \Phi_\rho(\lambda, \mu , x), 
\label{eq2}
\end{eqnarray}
where  the indeterminate ratio  $( 1 - 1^{ \lc x \rc}  ) / (1 -1)$  equals $\lc x \rc$ when $\mu=1$.
Moreover, the formula
$$
\phi_{\lambda,\mu,\delta, \rho}(y) = \lf y \rf + {\delta \over 1 - \lambda} - { \delta -\mu ( \lambda + \delta -1)\over \lambda} \Phi_\rho(\lambda, \mu , -\{ y\})
$$
holds for any real number $y$.
\end{lemma}

\begin{proof}
From Definition 4, we can write
$$
\begin{aligned}
\Phi_\rho(\lambda,\mu , x ) =  &\sum_{ k\ge 0} \lambda^k \sum_{ 0 \le l \le \lc k \rho + x \rc -1} \mu^l
=
 \sum_{k\ge 0} \lambda^k { 1 - \mu ^{ \lc k \rho + x\rc}\over 1 - \mu}
 \\
= &{ 1 \over (1-\lambda)(1-\mu)} - {1\over 1-\mu} \sum_{k\ge 0} \lambda^k \mu^{\lc k\rho + x\rc}
\end{aligned}
$$
which implies \eqref{eq1}. For equation \eqref{eq2}, multiplying \eqref{eq1}  by $ 1 -\lambda$, we find
$$
\begin{aligned}
(1-\lambda) \sum_{k\ge 0} \lambda^k \mu^{\lc k\rho + x\rc}
=& \sum_{k\ge 0} \lambda^k \mu^{\lc k\rho + x\rc}-\sum_{k\ge 1} \lambda^k \mu^{\lc (k-1)\rho + x\rc}
\\
 = &\mu^{\lc x \rc} + \lambda \sum_{k\ge 0} \lambda^k \left( \mu^{\lc (k+1)\rho +x \rc} - \mu^{\lc k \rho + x \rc } \right).
\end{aligned}
$$
Observe that, for any integer $k\ge 0$, $\lc (k+1)\rho +x \rc - \lc k \rho + x \rc $ takes only the value 0 or 1. Therefore
$$
\mu^{\lc (k+1)\rho +x \rc} - \mu^{\lc k \rho + x \rc } 
=(\mu-1)( \lc (k+1)\rho + x \rc - \lc k\rho + x \rc ) \mu^{\lc k \rho + x \rc }.
$$
We obtain the equality
$$
(1-\lambda)\left({ 1 \over 1 - \lambda} - (1-\mu)\Phi_\rho(\lambda,\mu , x)\right)
=\mu^{\lc x \rc} + \lambda(\mu-1)\sum_{k \ge 0 } ( \lc (k+1)\rho + x \rc - \lc k\rho + x \rc ) \lambda^k \mu^{\lc k\rho + x \rc}
$$
from which formula \eqref{eq2} follows. 

The map $y \mapsto \lf y \rf - \lf y - k \rho\rf $ has period 1 for any integer $k$. We can thus replace  $y$ by its fractional part $\{ y \}$ in the sum over $k$  occurring in the definition 3 giving $\phi_{\lambda,\mu,\delta, \rho}$. Observe also that
$\lf x \rf = -  \lc -x \rc$ for any real number $x$. We can therefore rewrite $\phi_{\lambda,\mu,\delta,\rho}(y)$ in the form
$$
\phi_{\lambda,\mu, \delta,\rho}(y) =  \lfloor y\rfloor   + {1-\delta\over \lambda} + \sum_{k= 0}^{+\infty}\lambda^k\mu^{\lc k\rho - \{ y\}\rc }\left( {\lambda +\delta -1 \over \lambda} - \lc (k+1)\rho -\{y\} \rc   +  \lc k\rho- \{ y\}\rc  \right).
$$
Using \eqref{eq1} and \eqref{eq2}  for $ x = -\{ y\}$ and noting that $\lc -\{ y\} \rc  = 0$, we obtain
$$
\begin{aligned}
\phi_{\lambda,\mu,\delta,\rho}(y) = &\lf y \rf + {1-\delta \over \lambda} + { \lambda + \delta -1 \over \lambda} \left( { 1 \over 1 - \lambda}
  - (1-\mu)\Phi_\rho(\lambda,\mu , -\{ y \})\right)
  \\
  & \hskip 6cm - {1-\lambda\over \lambda} 
\Phi_\rho(\lambda,\mu , -\{ y \})
\\
= & 
\lf y \rf + {\delta \over 1 - \lambda} - { \delta -\mu ( \lambda + \delta -1)\over \lambda} \Phi_\rho(\lambda, \mu , -\{ y\}).
\end{aligned}
$$
\end{proof}

\begin{corollary}
Let $0< \lambda <1$, $\mu >0$, $0< \rho <r_{\lambda,\mu}$ and $1-\lambda < \delta <d_{\lambda,\mu}$,
then the function $\phi_{\lambda,\mu,\delta, \rho}$  is right continuous and non-decreasing on the interval $I= [ 0, 1)$. Moreover, \\
(i) $\phi_{\lambda,\mu,\delta, \rho}$ is strictly increasing on $I$, if $\rho$ is irrational.\\
(ii) If $\rho = {p\over q}$ is rational, the function $y\mapsto \phi_{\lambda,\mu, \delta,p/q}(y)$ is constant on each interval
$[ {n\over q},{n+1\over q}), \, n\in \Z$.  \\
(iii) In any case, the relation $\phi_{\lambda,\mu,\delta,\rho}(y+1) = \phi_{\lambda,\mu,\delta, \rho}(y)+1$ holds for any real number $y$.
\end{corollary}

\begin{proof}
The function $x \mapsto \Phi_\rho(\lambda,\mu , x)$ is clearly non-decreasing and strictly increasing when $\rho$ is irrational. By Lemma 2.2, we have
$$
\phi_{\lambda,\mu,\delta, \rho}(y) =  {\delta \over 1 - \lambda} - { \delta -\mu ( \lambda + \delta -1)\over \lambda} \Phi_\rho(\lambda, \mu , -y)
$$
when $0 \le y <1$. Notice   that,  by the assumption, the coefficient $-{ \delta -\mu ( \lambda + \delta -1)\over \lambda} $ is negative which yields that $\phi_{\lambda,\mu,\delta,\rho}$  is non-decreasing. The other assertions are straightforward.

\end{proof}

\subsection{Proof of Theorem 2}
 Let us begin with the following result on the transcendency of values of the Hecke-Mahler function, due to Loxton and Van der Poorten  \cite{LoVdPA}. See also Sections 2.9 and 2.10 of the monograph \cite{Ni} and the survey article \cite{LoVdPB}.

\begin{theorem} 
Let $\lambda$ and $\mu$ be non-zero algebraic numbers and let $\rho$ be an irrational real number. Assume that  $0< | \lambda | <1$ and    $| \lambda | | \mu |^\rho <1$.
Then $\Psi_\rho(\lambda,\mu) $ is a transcendental  number. 
\end{theorem}


Using the homographic relations \eqref{eqdelta}, both  numbers $\delta(\lambda,\mu, \rho)$ and $\Psi_\rho(\lambda,\mu)$ are simultaneously either algebraic or transcendental. Then, it follows from  Theorem 4 that
 $\delta(\lambda,\mu, \rho)$ is a transcendental number for any irrational real number $0 < \rho <r_{\lambda,\mu}$.
As a consequence of the assertion (ii) of Theorem 1, the rotation number $\rho_{\lambda,\mu, \delta}$ cannot be an irrational number $\rho$ when $\lambda, \mu, \delta$ are algebraic numbers. It is therefore a rational number. 
Theorem 2 is established.

\section{Properties of the function $\phi$}

Let $\lambda, \mu, \delta, \rho$ be four real numbers satisfying the inequalities
$$
0< \lambda <1, \, \mu >0, \, 1-\lambda <\delta <d_{\lambda,\mu},\,  0<\rho<r_{\lambda,\mu}. 
$$
We estimate in this technical section the value of the function $\phi_{\lambda,\mu,\delta, \rho}$ at the points $0$ and $1-\rho$ according to the values of $\delta$.
We stress that $\rho$ is not assumed here to be the rotation number of the map $f_{\lambda,\mu, \delta}$. On the opposite, we shall make use of our results to identify this rotation number $\rho_{\lambda,\mu, \delta}$ in the subsequent Section 5,  and thus proving Theorem 1. 
Our estimates are based on numerical relations betweeen some special values of the Hecke-Mahler series $\Phi_\rho$ and the function $\sigma$, as for instance the formulae \eqref{eq3} to \eqref{eq6} below. 
 
\begin{lemma}
Assume that $0<\rho <r_{\lambda,\mu}$ is irrational. Let $\delta = \delta(\lambda,\mu,\rho)$ and $\phi = \phi_{\lambda, \mu,\delta,\rho}$. Then the following equalities hold
$$
\phi(0)=0 \quad \hbox{\rm and} \quad \phi(1-\rho) = { 1 - \delta\over \lambda} =\eta.
$$
\end{lemma}

\begin{proof}
Recall the formula
$$
\delta= {(1-\lambda)(1+\mu \sigma)\over 1+ (\mu-1)\sigma},\quad{\rm where}\quad
\sigma= \sum_{k\ge1}\left(  \left\lfloor (k+1) \rho \right\rfloor   -  \left\lfloor k \rho \right\rfloor   \right)\lambda^{k} \mu^{\lfloor k \rho\rfloor  }.
$$
Notice first that we have the equalities
\begin{multline}
\Phi_\rho(\lambda,\mu, 0) = \sum_{k\ge 1} \sum_{ 0\le l < k \rho} \lambda^k\mu^l
=\sum_{k\ge 1}\lambda^k + \sum_{k\ge 1} \sum_{ 1\le l < k \rho} \lambda^k\mu^l
= { \lambda\over 1-\lambda} + \Psi_\rho(\lambda, \mu)
\\
= {\lambda\over 1-\lambda}(1+ \mu \sigma), \label{phi0}
\end{multline}
the last one coming from Lemma 2.1 and  noting that the strict inequality  $l< k\rho$ is equivalent to $l\le k\rho$,  when  $\rho$ is irrational.
 It follows from Lemma 2.2 and \eqref{phi0} that
$$
\begin{aligned}
\phi(0) = & {\delta \over 1 - \lambda} - { \delta -\mu ( \lambda + \delta -1)\over \lambda} \Phi_\rho(\lambda, \mu , 0)
=
{\delta \over 1 - \lambda} - { \delta -\mu ( \lambda + \delta -1)\over 1-\lambda} (1+\mu\sigma)
\\
= & {\mu\Big( \delta (1  +  (\mu-1)\sigma) - (1-\lambda)(1+\mu\sigma)\Big) \over 1-\lambda}= 0.
\end{aligned}
$$

For the value $\phi(1-\rho)$, we compute $\Phi_\rho(\lambda,\mu , \rho-1)$ using \eqref{eq2}. 
Noting that $\lc \rho -1 \rc =0$,  we find
$$
\begin{aligned}
{1-\lambda \over \lambda} \Phi_\rho(\lambda,\mu , \rho-1)
=&
\sum_{k \ge 0 } ( \lc (k+1)\rho + \rho-1 \rc - \lc k\rho + \rho-1 \rc ) \lambda^k \mu^{\lc k\rho + \rho-1 \rc}
\\
=&{1\over \lambda\mu}\sum_{k \ge 1 } ( \lc (k+1)\rho  \rc - \lc k\rho  \rc ) \lambda^k \mu^{\lc k\rho  \rc}
\\
= & {1\over \lambda}\sum_{k \ge 1 } ( \lf (k+1)\rho  \rf - \lf k\rho  \rf ) \lambda^k \mu^{\lf k\rho  \rf} = {\sigma\over \lambda},
\end{aligned}
$$
since $\lc k \rho \rc = \lf k \rho \rf +1$ for any integer $k\ge 1$. 
Therefore
$$
\begin{aligned}
& \phi(1-\rho)  = {\delta\over 1 -\lambda} - {\delta- \mu(\lambda + \delta -1)\over \lambda}
\Phi_\rho( \lambda,\mu, \rho -1) 
={\delta\over 1 -\lambda} - {\delta- \mu(\lambda + \delta -1)\over \lambda}
{\sigma\over 1 -\lambda}
\\
& = { \delta(\lambda +(\mu-1)\sigma) -(1-\lambda)\mu\sigma\over \lambda(1-\lambda)}
={(1-\lambda)(1-\delta) + \delta(1+(\mu-1)\sigma) -(1-\lambda)(1+\mu\sigma)\over \lambda(1-\lambda)} 
\\
&= {1-\delta\over \lambda},
\end{aligned}
$$
since $ \delta = (1-\lambda)(1+ \mu\sigma)/(1+ (\mu-1)\sigma)$.
 \end{proof}
 
When $\rho$ is a rational  number $ p/ q $, the function $\phi_{\lambda,\mu,\delta,p/q}$ is constant on any  interval of the form $[ {n\over q }, {n+1\over q}), n\in \Z$,  and has
 a positive jump at the endpoints $\Z/q$. In this case, we have the analogous
 
 \begin{lemma}
 Assume that $\rho =p/ q$ and 
 $
 \delta\left(\lambda, \mu , \left( p/ q\right)^- \right)\le \delta <   \delta(\lambda, \mu , p /q).
 $
 Put $\phi = \phi_{\lambda,\mu, \delta, p/q}$. Then
 $$
\phi\left(-{1\over q}\right) = \phi(0^-) <0  \le \phi(0)
$$
and
$$
\phi\left({q-p-1\over q}\right)= \phi\left(\left(1-{p\over q}\right)^-\right) < {1 -\delta \over \lambda} \le  \phi\left(1-{p\over q}\right)  .
 $$
 \end{lemma}

\begin{proof}
Set $ \sigma = \sigma(\lambda, \mu , {p\over q})$ and $ \sigma^{-}= \sigma(\lambda,\mu,({p\over q})^{-})$. We first show that
\begin{equation}
\sigma = { S + \lambda^{q-1}\mu^{p-1}\over 1 -\lambda^q\mu^p}\quad \hbox{\rm and}\quad
\sigma^- = { S + \lambda^{q}\mu^{p-1}\over 1 -\lambda^q\mu^p},
\label{sigma}
\end{equation}
where we recall the notation 
$$
S  = \sum_{k=1}^{q-2}  \left( \left\lfloor (k+1) \frac{p}{q} \right\rfloor   -  \left\lfloor k \frac{p}{q} \right\rfloor   \right)\lambda^{k}\mu^{\lfloor k{p\over q}\rfloor  }
$$
from Theorem 1. By Definition 2, we have 
$$
 \sigma =  \sum_{k\ge 1} \left( \lf (k+1){p\over q}\rf -\lf k {p\over q}\rf \right)\lambda^k \mu^{\lf k {p\over q}\rf}.
 $$
 Observe that  $\lf (k+q){p\over q}\rf = \lf k{p\over q}\rf +p$. Splitting the above sum over $k$ according to the various classes of $k$ modulo $q$, we obtain the first formula
 $$
 \sigma = {1\over 1 -\lambda^q \mu^p}\sum_{k= 1}^q \left( \lf (k+1){p\over q}\rf -\lf k {p\over q}\rf \right)\lambda^k \mu^{\lf k {p\over q}\rf} = { S + \lambda^{q-1}\mu^{p-1}\over 1-\lambda^q\mu^p} .
 $$
 Similarly, we have
$$
 \begin{aligned}
 \sigma^-  & =  \sum_{k\ge 1} \left(\left \lf ((k+1)p/ q)^-\right\rf - \left \lf (k p/ q)^- \right\rf \right)\lambda^k \mu^{\lf (k p/ q)^-\rf}
 \\
  & = {1\over 1 -\lambda^q \mu^p}\sum_{k= 1}^q   \left(\left \lf ((k+1)p/ q)^-\right\rf - \left \lf (k p/ q)^- \right\rf \right)\lambda^k \mu^{\lf (k p/ q)^-\rf}
= { S + \lambda^{q}\mu^{p-1}\over 1-\lambda^q\mu^p} .
  \end{aligned}
 $$
 
 Now, we establish the formulae
\begin{eqnarray}
 \Phi_{p/q}(\lambda, \mu , 0) &=& {\lambda(1+ \mu \sigma^-)\over 1-\lambda}, 
 \label{eq3}
 \\
\Phi_{p/q}\left(\lambda, \mu , (-1)^+\right)  = 
   \Phi_{p/q}\left(\lambda, \mu ,- 1+ {1\over q} \right)
& = &  {\lambda \sigma \over 1-\lambda},
\label{eq4}
  \\
  \Phi_{p/q}\left(\lambda, \mu , {p\over q} -1\right) &=&  {\sigma^-\over 1-\lambda}, 
  \label{eq5}
 \\
  \Phi_{p/q}\left(\lambda, \mu , \left({p\over q}-1\right)^+ \right) = 
   \Phi_{p/q}\left(\lambda, \mu , {p\over q}-1+{1\over q} \right)
 & =& {\sigma\over 1 -\lambda}.
  \label{eq6}
  \end{eqnarray}
 To that purpose, we observe that the function $x \mapsto \Phi_{p/q}(\lambda, \mu,x)$ is constant on each interval $({n\over q},{n+1\over q}], n\in \Z$,  and we use formula \eqref{eq2}. Gathering  as above the various classes of  $k$ modulo $q$, we obtain the sums
 $$
 \begin{aligned}
{1-\lambda \over \lambda}\Phi_{p/q}(\lambda, \mu , 0)= &
\sum_{k \ge 0 } \left( \left\lc (k+1){p\over q} \right\rc - \left\lc k {p\over q}  \right\rc \right) \lambda^k \mu^{\lc k{p\over q} \rc}
\\
=&
{\sum_{k = 0 }^{q-1} \left( \left\lc (k+1){p\over q} \right\rc - \left\lc k {p\over q}  \right\rc \right) \lambda^k \mu^{\lc k{p\over q} \rc}
\over 1 -\lambda^q\mu^p}
\\
=& 
{1 + \mu\sum_{k = 1 }^{q-2} \left( \left\lf (k+1){p\over q} \right\rf - \left\lf k {p\over q}  \right\rf \right) \lambda^k \mu^{\lf k{p\over q} \rf}
\over 1 -\lambda^q\mu^p}
={1 + \mu S \over 1 -\lambda^q\mu^p} = 1 + \mu \sigma^-,
\end{aligned}
$$
since $ \left\lc k {p\over q}  \right\rc =  \left\lf k {p\over q}  \right\rf  +1$ for $1 \le k \le q-1$
and $\left\lc 0 {p\over q}  \right\rc =0$. Similarly, we have the equalities
$$ 
\begin{aligned}
{1-\lambda \over \lambda}\Phi_{p/q}(\lambda, \mu , {1\over q}-1)= &
\sum_{k \ge 0 } \left( \left\lc {(k+1)p+1\over q} \right\rc - \left\lc {k p+1\over q}  \right\rc \right) \lambda^k \mu^{\lc {kp+1\over q} \rc-1}
\\
=&
{\sum_{k = 0 }^{q-1} 
 \left( \left\lc {(k+1)p+1\over q} \right\rc - \left\lc {k p+1\over q}  \right\rc \right) \lambda^k \mu^{\lc {kp+1\over q} \rc -1}
\over 1 -\lambda^q\mu^p}
\\
=& 
{\lambda^{q-1}\mu^{p-1} + \sum_{k = 1 }^{q-2} \left( \left\lf (k+1){p\over q} \right\rf - \left\lf k {p\over q}  \right\rf \right) \lambda^k \mu^{\lf k{p\over q} \rf}
\over 1 -\lambda^q\mu^p}
\\
= & {\lambda^{q-1}\mu^{p-1} + S \over 1 -\lambda^q\mu^p} = \sigma,
\end{aligned}
$$
since $ \left\lc {k p+1\over q}  \right\rc =  \left\lc { k p\over q}  \right\rc   =  \left\lf {k p\over q}  \right\rf +1$ for $1 \le k \le q-1$. For the value $x={p\over q}-1$, we find
 $$
 \begin{aligned}
{1-\lambda \over \lambda}\Phi_{p/q}(\lambda, \mu , {p\over q}-1)= &
\sum_{k \ge 0 } \left( \left\lc (k+1){p\over q} +{p\over q}-1\right\rc - \left\lc k {p\over q} + {p\over q}-1 \right\rc \right) \lambda^k \mu^{\lc k{p\over q} +{p\over q}-1\rc}
\\
=&
{1\over \lambda \mu}\sum_{k \ge 1 } \left( \left\lc (k+1){p\over q} \right\rc - \left\lc k {p\over q}  \right\rc \right) \lambda^k \mu^{\lc k{p\over q} \rc}
\\
= &
{ (1+ \mu \sigma^-)-1 \over \lambda \mu} = {\sigma^-\over \lambda},
\end{aligned}
$$
taking  again the  computations used for $\Phi_{p/q}(\lambda, \mu , 0)$. Finally, we get
 $$
 \begin{aligned}
{1-\lambda \over \lambda}\Phi_{p/q}\left(\lambda, \mu , {p+1\over q}-1\right)= &
\sum_{k \ge 0 } \left( \left\lc {(k+2)p+1\over q} \right\rc - \left\lc {(k +1)p +1\over q}  \right\rc \right) \lambda^k \mu^{\lc {(k+1)p+1\over q} \rc-1}
\\
=&
{1\over \lambda }\sum_{k \ge 1 } \left( \left\lc{ (k+1)p+1\over q} \right\rc - \left\lc {k p+1\over q} \right\rc \right) \lambda^k \mu^{\lc {kp+1\over q} \rc-1}
 = {\sigma\over \lambda},
\end{aligned}
$$
by the above computation of $\Phi_{p/q}(\lambda, \mu , {1\over q}-1)$. The formulae \eqref{eq3} to \eqref{eq6} are established.

 We now use Lemma 2.2 in order to estimate values of $\phi$. We have
   $$
   \phi(0)= {\delta \over 1 - \lambda} - { \delta -\mu ( \lambda + \delta -1)\over \lambda} \Phi_{p/q}(\lambda, \mu , 0).
 $$
Then \eqref{eq3} yields
\begin{multline}
\qquad (1-\lambda)\phi(0) = 
 \delta - (\delta -\mu ( \lambda + \delta -1))(1+\mu \sigma^-) 
 \\
=
 \mu\left( \delta ( 1 + (\mu-1)\sigma^-)- (1-\lambda)(1+ \mu \sigma^-)\right)
= \mu(1 + (\mu-1)\sigma^-)\left( \delta - \delta\left(\lambda,\mu , ({p/ q})^-\right) \right),
 \label{eq7}
 \end{multline} 
 since $\delta(\lambda,\mu, (p/q)^-) = (1-\lambda) (1+ \mu \sigma^-)/(1 + (\mu-1)\sigma^-)$. Observe now that the factor $1 + (\mu-1)\sigma^-$ is always positive. Indeed, we deduce from Lemma 2.1 and its corollary that
 $$
 0 \le \sigma^- < {1-\lambda\over \lambda \mu} \lim_{\rho\nearrow r_{\lambda,\mu}}\Psi_\rho(\lambda,\mu) = \begin{cases}   {\lambda\over 1 -\lambda\mu}  \qquad  if \quad  \lambda \mu <1, 
\\
 +\infty   \qquad   if  \quad \lambda \mu \ge 1.
\end{cases} 
 $$
 Therefore $1 + (\mu-1)\sigma^-$ is bounded from below by $1 $ when $\mu \ge 1$ and by  $(1-\lambda)/( 1-\lambda\mu)$ when $\mu <1$.
 It follows that $\phi(0)$ is $\ge 0$  if and only if $ \delta \ge \delta(\lambda,\mu , (p/q)^-)$.

 For the value $\phi(0^-)$,  observe that $-\{ y\}=-1-y$ tends to $-1$ from above when 
 $y<0$ tends to 0. Lemma 2.2 provides now the formula
 $$
   \phi(0^-)= -1+ {\delta \over 1 - \lambda} - { \delta -\mu ( \lambda + \delta -1)\over \lambda} \Phi_{p/q}(\lambda, \mu , (-1)^+).
 $$
Then \eqref{eq4} yields
\begin{multline}
\qquad (1-\lambda)\phi(0^-) = 
 ( \lambda + \delta -1)(1+\mu \sigma) -\delta\sigma
 \\
= 
  \delta ( 1 + (\mu-1)\sigma)- (1-\lambda)(1+ \mu \sigma)
 =  (1 + (\mu-1)\sigma)\left( \delta - \delta\left(\lambda,\mu , p/ q\right) \right).
 \label{eq8}
 \end{multline}
 It follows that $\phi(0^-)$ is negative if and only if $\delta  < \delta\left(\lambda,\mu , p/ q\right)$.
 
 We now  deal with the lower bound at the point $1-p/q$. Using Lemma 2.2 and \eqref{eq5}, we find
 $$
\begin{aligned}
& \phi(1-{p\over q})  = {\delta\over 1 -\lambda} - {\delta- \mu(\lambda + \delta -1)\over \lambda}
\Phi_\rho( \lambda,\mu, {p\over q} -1) 
={\delta\over 1 -\lambda} - {\delta- \mu(\lambda + \delta -1)\over \lambda}
{\sigma^-\over 1 -\lambda}
\\
& = { \delta(\lambda +(\mu-1)\sigma^-) -(1-\lambda)\mu\sigma^-\over \lambda(1-\lambda)}
={(1-\lambda)(1-\delta) + \delta(1+(\mu-1)\sigma^-) -(1-\lambda)(1+\mu\sigma^-)\over \lambda(1-\lambda)} 
\\
& \ge {1-\delta\over \lambda},
\end{aligned}
$$
since the expression 
$$
 \delta(1+(\mu-1)\sigma^-) -(1-\lambda)(1+\mu\sigma^-)= (1+(\mu-1)\sigma^-)(\delta - \delta(\lambda,\mu,(p/q)^-)) 
 $$
 appearing above in the numerator is $\ge 0$. 
 
 The computations are similar for the left limit at the point $1-p/q$. Using \eqref{eq6},  we find
 $$
\begin{aligned}
& \phi((1-{p\over q})^-)  = {\delta\over 1 -\lambda} - {\delta- \mu(\lambda + \delta -1)\over \lambda}
\Phi_\rho( \lambda,\mu, ({p\over q} -1)^+) 
={\delta\over 1 -\lambda} - {\delta- \mu(\lambda + \delta -1)\over \lambda}
{\sigma\over 1 -\lambda}
\\
& = { \delta(\lambda +(\mu-1)\sigma) -(1-\lambda)\mu\sigma\over \lambda(1-\lambda)}
={(1-\lambda)(1-\delta) + \delta(1+(\mu-1)\sigma) -(1-\lambda)(1+\mu\sigma)\over \lambda(1-\lambda)} 
\\
& < {1-\delta\over \lambda},
\end{aligned}
$$
since
$$
 \delta(1+(\mu-1)\sigma) -(1-\lambda)(1+\mu\sigma)= (1+(\mu-1)\sigma)(\delta - \delta(\lambda,\mu,p/q)) < 0.
 $$
 
 \end{proof}

\section{The lift $F$}

 Let $F : \R \mapsto \R$ be the real function defined by
$$
F(x) = F_{\lambda,\mu,\delta}(x)= 
 \begin{cases} \lambda x + \delta  + (1-\lambda)\lfloor x\rfloor     \quad  if \quad 0\le \{x \}< \eta , 
\\
\mu(\lambda x + \delta-1)  +1  +(1-\lambda \mu) \lfloor x\rfloor       \quad   if  \quad \eta \le \{x\} <1.
\end{cases}
$$
Then $F $ is a \hbox{\it lift} of $f$, meaning  that $F$ satisfies the following properties:
\\
(i) For every $x\in \R$, we have
$$
\{F(x)\}=f(\{x\}).
$$
(ii) $F(x+1)=F(x)+1$, for every $x\in \R$.\\
(iii) $F$ is an increasing  function on $\R$ which is continuous on each interval of $\R \setminus \Z$ and right continuous everywhere.

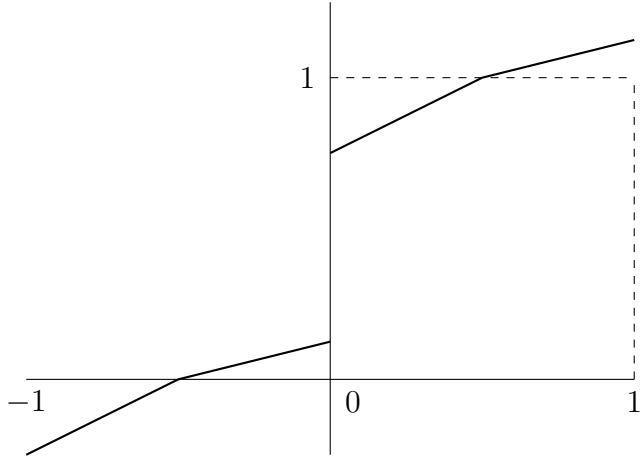
\begin{figure}[ht] 
\begin{tikzpicture}
\draw (-4,0) -- (4,0);
\draw (0,-1) -- (0,5);
\draw[dashed]  (4,0) -- (4,4);
\draw[dashed]  (0,4) -- (4,4);
\draw[thick]   (0,3) -- (2,4);
\draw[thick]   (2,4) -- (4, 4.5);
\draw[thick]   (-4,-1) -- (-2,0);
\draw[thick]  (-2,0) -- (0,0.5);
\node   at (0.3 ,-0.3) {$0$};
\node   at (-0.3,4) {$1$};
\node   at (-4,-0.3){$-1$};
\node   at (4,-0.3) {$1$};
\end{tikzpicture}
\caption{ Plot of $F_{1/2, 1/2, 3/4}(x) $ in the interval $-1\le x < 1$ }
\end{figure}

Let $x \in \R$ and let $(x_k)_{k\ge 0}$ be the {\it forward orbit}  of $x$ by $F$, 
where $x_k = F^k(x)$ stands for the $k$-th iterate of the function $F$.
When $\{x\}$ belongs to $C=\cap_{k\ge 1}f^k(I)$, we denote moreover  by $(x_{-k})_{k\ge 0}$ the  {\it backward orbit} of $x$ by $F$, where $x_{-k}= F^{-k}(x)$ is the $k$-th preimage of $x$ by $F$. This makes sense since $\{x_{-k}\} = f^{-k}(\{ x\})$ is the $k$-th preimage of $\{ x\}$ by $f$ and $x_{-k-1}= F^{-1}(x_{-k})$ is the inverse image of $x_{-k}$ by the injective map $F$, noting that $x_{-k}\in F(\R) = f(I)+\Z$ for all $k\ge 0$.

A fundamental property is that any forward orbit $(x_k)_{k\ge 0}$  can be computed  explicitely in terms of  its initial point $x$ and of the associated  {\it symbolic   sequence }  
$$
\lf x_{k+1}\rf - \lf x_k \rf \in \{ 0, 1\}. 
$$
  It turns out that, for any orbit, this symbolic sequence  is either periodic when the rotation number $\rho$ is rational, or a sturmian sequence of slope $\rho$  in the irrational case. We have the following explicit recursion formulae which motivate our definition of the conjugation $\phi$:

\begin{lemma}
 Let $x\in \R$ and let $(x_k)_{k\ge 0}$ be the forward orbit of $x$ by $F$.  For any non-negative integers $l$ and $n$, we have the relation
$$
\begin{aligned}
x_{l+n} = \lfloor x_{l+n}\rfloor   + \lambda^n\mu^{\lfloor x_{l+n}\rfloor  - \lfloor x_l\rfloor  } &\left( \{ x_l\}  - {1-\delta\over \lambda}\right)  + {1-\delta\over \lambda} +
\\
& \sum_{k=0}^{n-1} \lambda^k\mu^{\lfloor x_{l+n}\rfloor   -\lfloor  x_{l+n-k}\rfloor  }\left( {\lambda+ \delta -1\over \lambda} + \lfloor x_{l+n-(k+1)}\rfloor  -\lfloor x_{l+n-k}\rfloor  \right).
\end{aligned}
$$
\\
Moreover, assume that $\{x\} \in C$.  Let $(x_{-k})_{k\ge 0}$ be the backward orbit of $x$ by $F$ and  assume that there exist two real numbers $y$ and $\rho$ with $0< \rho <1$ such that $\lfloor x_k\rfloor   = \lfloor y+ k\rho\rfloor  $ for  all  integer $k\le 0$. Then, we have the series expansion
$$
x=\lfloor x\rfloor  + {1-\lambda\over \delta} +
\sum_{k\ge 0} \lambda^k\mu^{\lfloor x\rfloor   -\lfloor y-k\rho\rfloor  }\left( {\lambda+ \delta -1\over \lambda} + \lfloor y-(k+1)\rho\rfloor  -\lfloor y-k\rho\rfloor  \right)=\phi_{\lambda,\mu,\delta,\rho}(y).
$$
\end{lemma}

\begin{proof}
Notice that $\lfloor x_{k+1}\rfloor  -\lfloor x_k\rfloor  \in \{ 0,1\}$ for any integer $k$. 
Put
$$
\begin{aligned}
a_k := &\lambda \mu^{ \lfloor x_{k+1}\rfloor  -\lfloor x_k\rfloor  } = \begin{cases} \lambda  \quad  \hbox{\rm if} \quad   \lfloor x_k\rfloor  = \lfloor x_{k+1}\rfloor  ,
\\
 \lambda \mu  \quad   \hbox{\rm if } \quad \lfloor x_{k+1}\rfloor   =\lfloor x_k\rfloor  +1
\end{cases}
\\
b_k  :=&  1 + \mu^{\lfloor x_{k+1}\rfloor  -\lfloor x_k\rfloor  } (\delta -1) +(1 -\lambda\mu^{\lfloor x_{k+1}\rfloor  - \lfloor x_k\rfloor  })\lfloor x_k\rfloor   
\\
=  & \begin{cases} \delta +(1-\lambda)\lfloor x_k\rfloor   \quad  \hbox{\rm if} \quad   \lfloor x_k\rfloor  = \lfloor x_{k+1}\rfloor  ,
\\
\mu (\delta - 1) +1+(1-\lambda \mu) \lfloor x_k\rfloor   \quad   \hbox{\rm if } \quad \lfloor x_{k+1}\rfloor   =\lfloor x_k\rfloor  +1,
\end{cases}
\end{aligned}
$$
so that  $F(x_k)= a_k x_k + b_k$. Thus $x_{k+1}= a_kx_k+ b_k$. Let $n$ be a positive integer. Composing these  affine relations for $k= l, \dots , l+n-1$, we  obtain by induction on $n$  the formula
$$
x_{l+n} = a_{l+n-1} \cdots a_l x_l + \sum_{k=0}^{n-1}  a_{l+n-1} \cdots a_{l+k+1}  b_{l+k} 
=v_0 x_l + \sum_{k=0}^{n-1} v_{k+1} b_{l+k}, 
$$
where we have set 
$$
v_k : = a_{l+n-1} \cdots a_{l+k} = \lambda^{n-k} \mu^{ \lfloor x_{l+n}\rfloor   -\lfloor x_{l+k}\rfloor  } \quad \hbox{\rm for} \quad 
0 \le k \le n.
$$
We display the terms
$$
v_{k+1} b_{l+k}= v_{k+1} +{\delta -1\over \lambda} v_k + (v_{k+1}-v_k)\lfloor x_{l+k}\rfloor   
$$
appearing in the above sum. Note that   $v_n=1$. Thus, by Abel's summation, we find
$$
\begin{aligned}
x_{l+n} & = v_0 x_l + \sum_{k=0}^{n-1} v_{k+1} + {\delta -1\over \lambda} v_k + (v_{k+1}-v_k) \lf x_{l+k} \rf
\\
& = v_0\left( x_l-\lfloor x_l\rfloor  +{\delta -1\over \lambda}\right) + \lfloor x_{l+n-1}\rfloor   +1 + {\lambda + \delta -1\over \lambda} \sum_{k=1}^{n-1} v_k + \sum_{k=1}^{n-1} v_k(\lfloor x_{l+k-1}\rfloor   -\lfloor x_{l+k}\rfloor  )
\\
&= v_0\left( \{ x_l\}+{\delta -1\over \lambda}\right)+ \lfloor x_{l+n}\rfloor   + {1-\delta\over \lambda} + \sum_{k=1}^nv_k\left( {\lambda +\delta-1\over \lambda} + \lfloor x_{l+k-1}\rfloor  - \lfloor x_{l+k}\rfloor  \right).
\end{aligned}
$$
Replacing the index of summation $k$  by $n-k$ in the last sum, we find
$$
\begin{aligned}
\sum_{k=1}^nv_k & \left( {\lambda +\delta-1\over \lambda} + \lfloor x_{l+k-1}\rfloor  - \lfloor x_{l+k}\rfloor  \right)
= \sum_{k=0}^{n-1} v_{n-k}\left( {\lambda +\delta-1\over \lambda} + \lfloor x_{l+n-(k+1)}\rfloor  - \lfloor x_{l+n-k}\rfloor  \right)
\\
& =
\sum_{k=0}^{n-1} \lambda^k\mu^{\lfloor x_{l+n}\rfloor  -\lfloor x_{l+n-k}\rfloor  }\left( {\lambda +\delta-1\over \lambda} + \lfloor x_{l+n-(k+1)}\rfloor  - \lfloor x_{l+n-k}\rfloor  \right).
\end{aligned}
$$
The first assertion  is established.

For the second one, observe that the first assertion remains valid for $l$ negative when $\{x\}$ belongs to $C$, since then $x$ has a  preimage $F^{l}(x)$ and we apply the formula at this point. Choosing $l=-n$ and letting $n$ tend to infinity, we get the stated series expansion.

\end{proof}

The next result is crucial in our approach. It shows that $\phi$ satisfies a functional equation as in Theorem 3.
\begin{lemma}
Let $\lambda$, $\mu$, $\delta$ and $\rho$ be real numbers such that $0< \lambda <1, \mu >0$, $0< \rho< r_{\lambda,\mu}$.  We assume that $\delta = \delta(\lambda,\mu,\rho)$ when $\rho$ is irrational, or that $\delta$ belongs to the interval \eqref{eq12} 
 when  $\rho=p/ q$ is rational. Put $\phi= \phi_{\lambda,\mu,\delta,\rho}$ and $F=F_{\lambda,\mu,\delta}$. 
 Then,  the relations
$$
\lf \phi(y) \rf = \lf y\rf \qquad \hbox{\rm and } \qquad F( \phi(y)) = \phi(y+\rho)
$$
hold for any real number $y$. Thus, the $F$-orbit $(x_k)_{k\in \Z}$ of $x =\phi(y)$  is given by  the sequence
$$
x_k=\phi(y + k\rho), \, \forall k \in\Z.
$$
\end{lemma}

\eject

\begin{proof}
We first show that 
$0 \le \phi(y) <1$ when $0\le y <1$. 
In the case $\rho$ irrational, Lemma 3.1 gives
$\phi(0)= 0$. Thus $\phi(1) = \phi(0)+1 =1$ and the corollary of Lemma 2.2 asserts that the function $\phi$ is strictly increasing. Therefore 
$$
0 =\phi(0) \le \phi(y) <\phi(1) =1.
$$ 
In the rational case, the function $\phi$ is non-increasing and constant on each interval $[{n\over q},{n+1\over q}), n\in \Z$. 
Now,  we know that $\phi(0) \ge 0 $ and $\phi(0^-) <0$ by Lemma 3.2. Therefore 
$$
0 \le\phi(0) \le \phi(y) \le \phi(1^-) =\phi(0^-)+1 <1.
$$
For any $y\in \R$, we can write
$$
\phi(y) = \phi( \lf y \rf + \{ y\}) = \lf y \rf + \phi( \{y\}).
$$
We have thus proved that $\lf \phi(y) \rf  =\lf y \rf $ and $\{\phi(y)\} = \phi(\{ y\})$ for all real number $y$.

We now prove the relation $F( \phi(y)) = \phi(y+\rho)$.  By definition of $F$, we have to deal with two expressions for the value of $F( \phi(y)) $ depending whether the fractional part $\{ \phi(y)\}$ is smaller than $\eta = (1 -\delta )/ \lambda $ or not. But $\{ \phi(y)\} = \phi(\{ y \})$ and  Lemmae 3.1 and 3.2 yield that   $\phi(\{ y\})$ belongs the  interval $[0,\eta)$ when $\{ y\} < 1-\rho$,  and  to the other  interval $[\eta,1)$ when $\{ y\} \ge 1-\rho$, since $\phi$ is non-decreasing. The computation splits into two cases. 

Suppose first that $\{ y \} < 1 -  \rho$. Then $ \lf y+ \rho \rf = \lf y \rf$. Moreover,  $\{ \phi(y)\} = \phi(\{ y \}) < \eta $ by Lemmae 3.1 and 3.2 and the increasing monotonicity of the function $\phi$. Using the expression of $F$ in the intervals $[n, n+\eta), n\in\Z$, we obtain the equalities:
$$
\begin{aligned}
 F(\phi(y) ) & = \lambda \phi(y) + \delta + (1-\lambda) \lf \phi(y) \rf  
 = \delta + (1-\lambda) \lf y\rf 
 \\
 &  + \lambda \left(  \lfloor y\rfloor   + {1-\delta\over \lambda} + \sum_{k= 0}^{+\infty}\lambda^k\mu^{\lfloor  y\rfloor   -\lfloor y-k\rho\rfloor  }\left( {\lambda +\delta -1 \over \lambda} + \lfloor y-(k+1)\rho \rfloor   - \lfloor y-k\rho\rfloor  \right) \right)
\\
& = 
\lf y \rf + 1 + \sum_{k= 0}^{+\infty}\lambda^{k+1}\mu^{\lfloor  y\rfloor   -\lfloor y-k\rho\rfloor  }\left( {\lambda +\delta -1 \over \lambda} + \lfloor y-(k+1)\rho \rfloor   - \lfloor y-k\rho\rfloor  \right) 
\\
& = 
\lf y \rf + 1 + \sum_{k= 1}^{+\infty}\lambda^{k}\mu^{\lfloor  y\rfloor   -\lfloor y-(k-1)\rho\rfloor  }\left( {\lambda +\delta -1 \over \lambda} + \lfloor y-k\rho \rfloor   - \lfloor y-(k-1)\rho\rfloor  \right) 
\\
& = 
\lf y +\rho\rf + 1  
\\
&\qquad \qquad + \sum_{k= 1}^{+\infty}\lambda^{k}\mu^{\lfloor  y+ \rho \rfloor   -\lfloor y+\rho-k\rho\rfloor  }\left( {\lambda +\delta -1 \over \lambda} + \lfloor y+ \rho-(k+1)\rho \rfloor   - \lfloor y+\rho -k\rho\rfloor  \right) 
\\
& = 
\lf y +\rho\rf + 1 - {\lambda + \delta -1\over \lambda} 
\\
& \qquad \qquad 
+ \sum_{k= 0}^{+\infty}\lambda^{k}\mu^{\lfloor  y+ \rho \rfloor   -\lfloor y+\rho-k\rho\rfloor  }\left( {\lambda +\delta -1 \over \lambda} + \lfloor y+ \rho-(k+1)\rho \rfloor   - \lfloor y+\rho -k\rho\rfloor  \right) 
\\
& = \phi(y+ \rho).
\end{aligned}
$$

The case $\{ y\} \ge 1 -\rho$ is similar. Then $\lf y + \rho \rf = \lf y \rf + 1$ and $\{\phi(y)\} \ge \eta $ by Lemmae 3.1 and 3.2. We now use the expression of $F$ in the intervals $[n+\eta, n+1), n\in\Z$. We then obtain the equalities:

$$
\begin{aligned}
 F(\phi(y) ) & = \lambda \mu \phi(y) + \mu \delta +1 -\mu + (1-\lambda\mu ) \lf \phi(y) \rf  
 = \mu \delta +1 -\mu +  (1-\lambda\mu ) \lf y\rf 
 \\
 &  + \lambda \mu \left(  \lfloor y\rfloor   + {1-\delta\over \lambda} + \sum_{k= 0}^{+\infty}\lambda^k\mu^{\lfloor  y\rfloor   -\lfloor y-k\rho\rfloor  }\left( {\lambda +\delta -1 \over \lambda} + \lfloor y-(k+1)\rho \rfloor   - \lfloor y-k\rho\rfloor  \right) \right)
\\
& = 
\lf y \rf + 1 + \sum_{k= 0}^{+\infty}\lambda^{k+1}\mu^{\lfloor  y\rfloor  +1  -\lfloor y-k\rho\rfloor  }\left( {\lambda +\delta -1 \over \lambda} + \lfloor y-(k+1)\rho \rfloor   - \lfloor y-k\rho\rfloor  \right) 
\\
& = 
\lf y \rf + 1 + \sum_{k= 1}^{+\infty}\lambda^{k}\mu^{\lfloor  y\rfloor  +1  -\lfloor y-(k-1)\rho\rfloor  }\left( {\lambda +\delta -1 \over \lambda} + \lfloor y-k\rho \rfloor   - \lfloor y-(k-1)\rho\rfloor  \right) 
\\
& = 
\lf y +\rho\rf   +
 \sum_{k= 1}^{+\infty}\lambda^{k}\mu^{\lfloor  y+ \rho \rfloor   -\lfloor y+\rho-k\rho\rfloor  }\left( {\lambda +\delta -1 \over \lambda} + \lfloor y+ \rho-(k+1)\rho \rfloor   - \lfloor y+\rho -k\rho\rfloor  \right) 
\\
& = 
\lf y +\rho\rf - \left({\lambda + \delta -1\over \lambda} -1\right)
\\
& \qquad \qquad 
+ \sum_{k= 0}^{+\infty}\lambda^{k}\mu^{\lfloor  y+ \rho \rfloor   -\lfloor y+\rho-k\rho\rfloor  }\left( {\lambda +\delta -1 \over \lambda} + \lfloor y+ \rho-(k+1)\rho \rfloor   - \lfloor y+\rho -k\rho\rfloor  \right) 
\\
& = \phi(y+ \rho).
\end{aligned}
$$
\end{proof}

\section{Proof of Theorem 1}
Let $x \in \R$ and let $(x_k)_{k\ge 0}$ be the forward orbit of $x$ by $F$.  It is known (see \cite{RT}) that the limit 
$$
\rho_{\lambda,\mu , \delta} = \lim_{k \to \infty} \frac{x_k}{k} 
$$ 
exists and does not depend on the initial point $x$.  The number $\rho_{\lambda,\mu, \delta}$ is called   the {\it rotation number} of the map $f=f_{\lambda,\mu , \delta}$. 

Fix $\lambda$ and $\mu$ with $0< \lambda <1$, $\mu >0$. 
Let $\delta$ be a real number in the interval  $1-\lambda < \delta <d_{\lambda,\mu}$.
By the corollary of Lemma 2.1, the following alternative holds. Either $\delta$ belongs to the image of the interval $0< \rho<r_{\lambda,\mu}$ by the function $\rho \mapsto \delta(\lambda,\mu,\rho)$, or 
$$
 \delta\left(\lambda,\mu,(p/ q)^-\right) \le  \delta < \delta\left(\lambda,\mu , p/ q\right)
$$
for some rational number $p/q$ with $0< p/q <1$ (these intervals are the jumps of the increasing function $\rho \mapsto \delta(\lambda,\mu,\rho)$). In the latter case, Lemma 4.2 yields that $\rho_{\lambda,\mu,\delta}= p/q$. Indeed, we select an initial point  $x$ of the form  $x = \phi_{\lambda,\mu,\delta,p/q}(y) $ for an arbitrary  $y\in \R$, so that $x_k =  \phi_{\lambda,\mu,\delta,p/q}(y + k p/q) $ for every integer $k\ge 0$. Then,
$$
\rho_{\lambda,\mu , \delta} = \lim_{k \to \infty} \frac{x_k}{k} =  \lim_{k \to \infty} \frac{\lf x_k\rf}{k} =  \lim_{k \to \infty} \frac{\lf y + k{p\over q}\rf }{k}= {p\over q}.
$$

It remains to deal with parameters $\delta$ in the image, in other words $\delta = \delta(\lambda,\mu,\rho)$ for some $0< \rho <r_{\lambda,\mu}$. When $\rho$ is irrational, Lemma 4.2 yields as well  that  $\rho_{\lambda,\mu,\delta}= \rho$. When $ \rho = p/q$ is rational, we may use a general  argument of continuity. Proposition 5.7  in \cite{RT2} tells us that
the rotation number $\rho_{\lambda,\mu,\delta}$ is a continuous function of the parameter $\delta$. Thus
$$
\rho_{\lambda,\mu,\delta} = \lim_{\delta'\nearrow \delta}\rho_{\lambda,\mu,\delta'} = {p\over q},
$$
since we have already proved that $\rho_{\lambda,\mu,\delta'}$ is constant and equal to $ p/q$ when $\delta'$ is located in the right open  interval
$$
 \delta\left(\lambda,\mu,(p/ q)^-\right) \le  \delta' < \delta\left(\lambda,\mu , p/ q\right).
 $$
 
 We express now  $\delta(\lambda,\mu,p/q)$ and $ \delta\left(\lambda,\mu,(p/ q)^-\right) $ in term of the finite sum $S$. Recalling  formula \eqref{sigma},  we obtain
 $$
 \delta(\lambda,\mu,p/q)= {(1-\lambda) (1+ \mu \sigma)\over 1 + (\mu -1)\sigma}= 
{ (1-\lambda)(1 +\mu S + \lambda^{q-1}\mu^p(1-\lambda)) \over 1 +(\mu-1)S + \lambda^{q-1}\mu^{p-1}(\mu-\lambda\mu-1) }
$$
and 
$$
\delta\left(\lambda,\mu,(p/ q)^-\right)  = {(1-\lambda) (1+ \mu \sigma^-)\over 1 + (\mu -1)\sigma^-}= 
{(1-\lambda) (1+ \mu S)\over 1 +(\mu -1)S - \lambda^q\mu^{p-1}}.
$$

\section{Irrational rotation number}

We prove  part (i) of Theorem 3  and we give furthermore a description of the iterated images $f^n(I)$ when the rotation number $\rho$ is irrational.

In this case, the function $\phi$ is strictly increasing  on $\R$ and jumps at the points $l \rho + \Z, l \ge 1$. Put 
$$
\xi_l =  \phi(\{l\rho\}),  \quad{\rm and} \quad  \xi_l^- = \phi(\{l \rho\}^-), \quad l\ge 1.
$$
 All the intervals $[\xi_l^- , \xi_l), l\ge 1$,  are pairwise disjoint and contained in $I=(0,1)$.
 
\begin{proposition}
For any integer $n\ge 1$, we have the decomposition into disjoint intervals
$$
f^n(I)=  I \setminus\bigcup_{l= 1}^n[\xi_l^-,\xi_l), 
$$
and the formulae
$$
\begin{aligned}
\xi_l =&  f^l(0)  ={1-\delta\over \lambda}+  \sum_{k=0}^l \lambda^{l-k}\mu^{\lf l \rho \rf  - \lf k\rho\rf}\left( {\lambda+\delta -1\over \lambda} + \lf (k-1)\rho \rf - \lf k\rho\rf \right),
\\
 \xi_l^- =&  f^l(1^-)  = f^l(0) - \lambda^{l-1}\mu^{\lf l\rho\rf} ( \delta - \mu(\lambda + \delta -1)).
\end{aligned}
$$
Moreover, the set equalities 
$$
C := \bigcap_{n\ge1}f^n(I)=  \phi(I) = I \setminus\bigcup_{[0,1]\ge  1}[\xi_l^-,\xi_l)
\quad\hbox{\rm and} \quad
\overline{C} =\overline{ \phi(I)} = [0,1] \setminus\bigcup_{l\ge  1}(\xi_l^-,\xi_l)
$$
 hold. The set $\overline{C}$ is topologically homeomorphic to a Cantor set. 
\end{proposition}

\begin{proof}
Lemma 4.2 shows that for any $ k\ge 0$ and any $l\ge 1$, we have the equalities
$$
\begin{aligned}
F^k(\xi_l) = \phi(\{ l \rho\}+ k \rho) = &\phi( (k+l)\rho -\lf l\rho\rf) = \phi( \{ (k+l)\rho\} +\lf (k+l)\rho\rf - \lf l\rho\rf) 
\\
=&  \xi_{k+l} + \lf (k+l)\rho\rf -\lf l\rho\rf =  \xi_{k+l} + \lf \{ l\rho\} + k\rho\rf
\end{aligned}
$$
and
$$
\begin{aligned}
F^k(\xi_l^-) = \phi(\{ l \rho\}^-+ k \rho) = &\phi( ((k+l)\rho)^-  - \lf l\rho\rf) = \phi( \{ (k+l)\rho\}^- +\lf (k+l)\rho\rf - \lf l\rho\rf) 
\\
=&  \xi_{k+l}^- + \lf (k+l)\rho\rf -\lf l\rho\rf = \xi_{k+l}^- + \lf \{ l\rho\} + k\rho\rf,
\end{aligned}
$$
where $F^k$ stands for the $k$-th iterate of $F$.
Since $F$ is increasing and continuous on $\R \setminus\Z$, it follows  that
\begin{equation}
F^k([\xi_l^-,\xi_l)) = [\xi_{k+l}^-,\xi_{k+l})+ \lf \{ l\rho\} + k\rho\rf , \label{eq9}
\end{equation}
so that any number  $z\in F^k([\xi_l^-,\xi_l)) $ has  integer part $\lf z\rf =  \lf \{ l\rho\} + k\rho\rf$.

We first show that
$$
\phi(I)= I \setminus\bigcup_{l\ge  1}[\xi_l^-,\xi_l).
$$
Since $\phi$ is right continuous and increasing, no point of $\phi(I)$  is located in an interval of the form $[\xi_l^-,\xi_l), l \ge 1$.  Thus, we have the inclusion
$
\phi(I) \subseteq I \setminus\bigcup_{l\ge  1}[\xi_l^-,\xi_l).
$
The reversed inclusion $
\phi(I) \supseteq I \setminus\bigcup_{l\ge  1}[\xi_l^-,\xi_l)
$
follows straightforwardly from the right continuity of $\phi$. Indeed,
let $x\in I$ which is located outside the intervals $[\xi_l^-,\xi_l), l \ge 1$. For every $n\ge 1$, define an  index $l_n$ among the integers $1\le l \le n$ for which  $x \le\xi_l$ and $\xi_{l_n}$ is the closest to $x$. It is readily seen that the decreasing  sequence $\{l_n\rho\}$ converges to a number $y$ and that $x= \phi(y)$ by right continuity of $\phi$.

We know by Lemma 3.1  that the critical point $\eta= \phi(1-\rho)$ is located in the image $\phi(I)$. 
In particular, this critical point $\eta$ does not belong to any interval $[\xi_l^-,\xi_l), l \ge 1$. The function $f$ is thus continuous on each interval $ [ \xi_l^-,\xi_l)$, so  that we deduce from \eqref{eq9} that 
\begin{equation}
 f([\xi_l^-,\xi_l)) = [\xi_{l+1}^-,\xi_{l+1})
 \label{eq10}
 \end{equation}
 by reducing modulo 1.
Now, Lemmae 3.1 and 4.2  yield the equalities
$$
\xi_1 = \phi(\rho) = F(\phi(0)) = F(0) = \delta \quad\hbox{\rm and} \quad \xi_1^- =\phi(\rho^-) = F(0^-) =  \mu(\lambda+ \delta-1).
$$
Looking at Figure 1, we immediately observe that  $ f(I) = I \setminus [\xi_1^-,\xi_1)$, as announced.  Taking now the image by $f$ and using \eqref{eq10}, we find  
$$
f^2(I) = f(I) \setminus f([\xi_1^-,\xi_1))=  I \setminus ( [\xi_1^-,\xi_1)\cup [\xi_2^-,\xi_2)).
$$
Arguing  by induction on $n\ge 1$, we thus deduce from \eqref{eq10} the required equality  
$$
f^n(I)=  I \setminus\bigcup_{l= 1}^n[\xi_l^-,\xi_l).
$$
Letting $n$ tend to infinity, we finally obtain that 
$$
 \bigcap_{n\ge1}f^n(I)= I \setminus\bigcup_{l\ge  1}[\xi_l^-,\xi_l) = \phi(I).
$$

It remains to establish  the explicit formulae giving $\xi_l$ and $\xi_l^-$. Lemma 4.1 delivers the expression
$$
\begin{aligned}
\phi(l\rho) =& \lf l\rho\rf + \lambda^l\mu^{\lf l \rho\rf}\left(\{0 \}- { 1-\delta \over \lambda} \right)
+  { 1-\delta \over \lambda} 
\\
& \qquad \qquad + \sum_{k=0}^{l-1} \lambda^k \mu^{ \lf l\rho\rf -\lf (l-k)\rho\rf} \left( {\lambda +\delta -1\over \lambda} +\lf (l-k-1)\rho\rf -\lf (l-k)\rho\rf\right)
\\
= &  \lf l\rho\rf +
{ 1-\delta \over \lambda} (1- \lambda^l\mu^{\lf l \rho\rf}) + \sum_{k=1}^{l} \lambda^{l-k }\mu^{ \lf l\rho\rf -\lf k\rho\rf} \left( {\lambda +\delta -1\over \lambda} +\lf (k-1)\rho\rf -\lf k\rho\rf\right)
\\
= &  \lf l\rho\rf +
{ 1-\delta \over \lambda} + \sum_{k=0}^{l} \lambda^{l-k }\mu^{ \lf l\rho\rf -\lf k\rho\rf} \left( {\lambda +\delta -1\over \lambda} +\lf (k-1)\rho\rf -\lf k\rho\rf\right),
\end{aligned}
$$
since $\lf \phi(k \rho)\rf= \lf k \rho\rf$ for every integer $k$. Taking the fractional part, we obtain the formula
$$
\xi_l = \phi(\{l \rho\}) = \{ \phi(l\rho)\}= {1-\delta\over \lambda}+  \sum_{k=0}^l \lambda^{l-k}\mu^{\lf l \rho \rf  - \lf k\rho\rf}\left( {\lambda+\delta -1\over \lambda} + \lf (k-1)\rho \rf - \lf k\rho\rf \right).
$$
The equality $\xi_l = f^l(0)$ immediately follows from Lemmae 3.1 and 4.2.

Similarly $f^l(1^-) = \xi_l^- = \{ \phi(l \rho^-)\}$. We have $\{ 0^-\} = 1, \lf 0^-\rf = -1$ and $\lf k \rho^-\rf
= \lf k \rho\rf$ for any integer $k\ge 1$. Then, Lemma 4.1 gives
$$
\begin{aligned}
\phi(l\rho^-) =& \lf l\rho\rf +  \lambda^l\mu^{\lf l \rho\rf + 1}\left( 1 - {1-\delta\over \lambda}\right) + {1-\delta\over \lambda} + \lambda^{l-1}\mu^{ \lf l\rho\rf} \left( {\lambda+ \delta-1\over \lambda} -1\right)
\\
&\qquad\qquad + \sum_{k=0}^{l-2} \lambda^k \mu^{ \lf l\rho\rf -\lf (l-k)\rho\rf} \left( {\lambda +\delta -1\over \lambda} +\lf (l-k-1)\rho\rf -\lf (l-k)\rho\rf\right)
\\
= &  \phi(l\rho) - \lambda^{l-1}\mu^{\lf l\rho\rf} ( \delta - \mu(\lambda + \delta -1)).
\end{aligned}
$$

As a corollary of the above formulae, let us briefly prove that $\overline{C}$ is a Cantor set.  The Lebesgue measure of 
$\overline{C} = I \setminus\bigcup_{l\ge  1}(\xi_l^-,\xi_l)$ is equal to
$$
1 - \sum_{l\ge 1} \xi_l - \xi_l^- = 1 -\big( \delta - \mu(\lambda + \delta -1)\big)\sum_{l\ge 1} \lambda^{l-1}\mu^{\lf l\rho\rf}.
$$
Using \eqref{eq1} and \eqref{phi0}, we  easily compute the sum
$$
 \sum_{l\ge 1} \lambda^{l-1}\mu^{\lf l\rho\rf} = { 1+ (\mu-1) \sigma\over 1 -\lambda},
 $$
 where $ \sigma = \sigma(\lambda,\mu,\rho)$. On the other hand, we have
 $$
\delta = {(1-\mu)(1+ \mu\sigma )\over 1 + (\mu-1)\sigma} \quad{\rm so \,\, that}\quad  \delta - \mu(\lambda + \delta -1)= {1-\lambda \over 1 + (\mu-1)\sigma}.
$$
Therefore $\overline{C}$ is a null set. Consequently, it has no inner point. Moreover, $\overline{C} = \overline{\phi(I)}$ has no isolated point, since the function $\phi$ is strictly increasing and right continuous. It follows that the compact set $\overline{C}$ is homeomorphic to a Cantor set.

\end{proof}

In order to complete the proof of  the assertion (i)  of   Theorem 3, we now show that for any point $x \in I$, the $\omega$-limit set $\omega(x)$ coincides with $\overline{\phi(I)}$. To that purpose, we consider  the $F$ orbit $x_k:= F^k(x) , k\ge 0$,  of $x$. 
When $x = \phi(y)$ belongs to
$\phi(I)$, Lemma 4.2 shows that $x_k=\phi(y+ k \rho)$ so that $f^k(x) = \{ x_k\} = \phi(\{y+ k \rho\})$. The sequence of fractional parts $(\{ y+ k \rho\})_{k\ge 0}$ is dense in $I$, since $\rho$ is an irrational number. Thus the set $\omega(x)$ of accumulation points of the orbit $(f^k(x))_{k\ge 0}$ is equal to $\overline{\phi(I)}$. 
It remains to deal with  points $x\in I$ not belonging to $\phi(I)$, it means $x \in [\xi_l^-,\xi_l)$ for some $l\ge 1$. Observe first that
$$
\lf x_k \rf = \lf \{ l \rho\}  + k \rho \rf, \, \forall k \ge 0,
$$
since $x_k \in F^k([\xi_l^-, \xi_l))$.
In particular the symbolic sequence $(\lf x_{k+1}\rf - \lf x_k \rf)_{k\ge 0}$ coincides with the sturmian sequence $(\lf y+ (k+1)\rho\rf - \lf y +k \rho\rf)_{k\ge 0}$, where we have set $y = \{l \rho\}$. Then, Lemma 4.1, with $l = 0$ and $n=k$,  provides us the  formula
$$
\begin{aligned}
x_{k} = \lfloor y+ k\rho \rfloor  & + \lambda^k\mu^{\lfloor y + k\rho \rfloor   } \left( x  - {1-\delta\over \lambda}\right)  +  {1-\delta\over \lambda} +
\\
& \sum_{j=0}^{k-1} \lambda^j\mu^{\lfloor y+k\rho\rfloor   -\lfloor y+k\rho -j \rho\rfloor  }\left( {\lambda+ \delta -1\over \lambda} + \lfloor y+k\rho -(j+1)\rho \rfloor  -\lfloor y+k\rho -j \rho\rfloor  \right)
\\
= \phi(y+ k &\rho) +  \lambda^k\mu^{\lf y+k\rho\rf} (x-\xi_l).
\end{aligned}
$$
Thus, the $F$-orbit $(x_k)_k$ converges exponentially fast to the $F$-orbit $(\phi(y+ k\rho))_{k\ge0}$
as $k$ tends to infinity. Reducing modulo 1, one obtains as well that $\omega(x)= \overline{\phi(I)}$.

 \section{Rational rotation number}

We prove the statement (ii) of Theorem 3 and add a dynamical description  of the iterated images of $f= f_{\lambda,\mu,\delta}$. We assume throughout this section that
the inequalities \eqref{eq12} are fulfilled for some rational number $p/q$.
Then, Theorem 1 asserts that the rotation number of $f$ equals $p/q$. Put $\phi= \phi_{\lambda,\mu,\delta,p/q}$ and set 
$$
\zeta_{m}=\phi  \left( { m \over q } \right) , \quad \mbox{for} \quad 0\le m\le q-1.
 $$
 It is convenient to extend the sequence $(\zeta_m)_{0\le m \le q-1}$ by $q$-periodicity setting  $\zeta_m=\zeta_r$ for any integer $m\in \Z$, where 
$r$ is the remainder in the euclidean division of $m$ by $q$. Then, Lemma 4.2 yields the formula
$$
f(\zeta_m) = \zeta_{m+p}, \quad \forall m\in \Z.
$$
As $\phi$ is nondecreasing $\zeta_0\le  \cdots \le\zeta_{q-1}$, we claim that these numbers are  distinct. If not, there exists $m$ such that $\zeta_m=\zeta_{m+1}$. Iterating $f$ thus $\zeta_{m+kp}=\zeta_{m+kp+1}$, $\forall k\in \Z$, obtaining 
that the function $\phi$ is constant, in contradiction for instance with Lemma 3.2. Moreover, $\zeta_{q-1} = \phi(1^-) = \phi(0^-)+1 <1$,
 again   by Lemma 3.2. So, we have  an  increasing  sequence
 $$
  0 \le \zeta_0< \ldots < \zeta_{q-1}<1
 $$
in $I$. It follows that the set $\phi(I) = \{ \zeta_0, \dots , \zeta_{q-1}\}$ is an $f$-cycle of order $q$, on which $f$ acts by the substitution  $m\mapsto m+p$  modulo $q$.

  Recall that  $f(\eta)=0$ and $f(\eta^-)=1$, where $ \eta= (1-\delta)/ \lambda$ is the critical point of the map $f$.  Moreover, Lemma 3.2 shows that $\zeta_{q-p-1}<\eta \le  \zeta_{q-p}$.  
 If $\eta$ does not belong to $\phi(I)$, we have strict inequalities $\zeta_{q-p-1}<\eta < \zeta_{q-p}$ and $f(\zeta_{q-p})= \zeta_0>0$.  Otherwise $\eta = \zeta_{q-p}$ and $\zeta_0=f(\eta)=0$. The latter case occurs only 
 when $ \delta$ coincides with the left end point of the interval \eqref{eq12}. Indeed  \eqref{eq7} shows that $\phi(0) $ vanishes if and only if 
 $\delta = \delta\left(\lambda,\mu,(p/ q)^-\right)$.
 
 \marginsize{2.5cm}{2.5cm}{1cm}{2cm}
 \begin{figure}[ht] 
 \begin{tikzpicture}[scale = 1.1]
\draw[thick]  (0,0) -- (6,0);
\draw[thick]  (8,0) -- (14,0);

\draw  (0,0) node[above ] {$0$} node{$\bullet$};
\draw  (1,0) node[above] {$\zeta_{0}$} node{$\bullet$};
\draw  (2,0) node[above ] {$\zeta_{1}$} node{$\bullet$};
\draw  (3,0) node[above ] {$\zeta_{q-p-1}$} node{$\bullet$};
\draw  (3.5,0) node[below ] {$\eta$} node{$\bullet$};
\draw  (4,0) node[above ] {$\zeta_{q-p}$} node{$\bullet$};
\draw  (5,0) node[above ] {$\zeta_{q-1}$} node{$\bullet$};
\draw  (6,0) node[above] {1} (6,0) circle (0.08);

\draw  (8,0) node[above  ] {$\zeta_0=0$} node{$\bullet$};
\draw  (9.5,0) node[above] {$\zeta_{1}$} node{$\bullet$};
\draw  (10.5,0) node[above ] {$\zeta_{2}$} node{$\bullet$};
\draw  (11.5,0) node[above ] {$\cdots$};
\draw  (13,0) node[above ] {$\zeta_{q-1}$} node{$\bullet$};
\draw  (14,0) node[above] {1} (14,0) circle (0.08);

\end{tikzpicture}
\caption{Case $\zeta_0>0$ and Case $\zeta_0=0$.}
\end{figure}
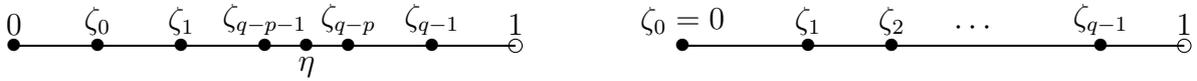

 The next proposition provides a partition of the image $f^n(I)$, $n\ge 1$,  into disjoint intervals. It is convenient to consider circular intervals (or circle arcs identifying $I$ with $\R/\Z$). For any $a , b $ both belonging  to $I$, we set
 $$
 [a,b) = 
 \begin{cases}
{ \rm the  \, usual \,  interval } \,\,     [a, b) & {\rm if} \,  a < b
 \\
  [0,b) \cup [a, 1) & {\rm  if} \,  a  > b.
 \end{cases}
 $$
 We  write for instance $[\zeta_{q-1},\zeta_0) = [0,\zeta_0)\cup [\zeta_{q-1},1)$.
 
 \begin{proposition}
 For any integer $l\ge 1$, the circular interval $\big[f^l(1^-), f^l(0)\big)$ is contained in $\big[\zeta_{l p-1}, \zeta_{l p}\big)= \big[f^l(\zeta_{q-1}),f^l(\zeta_0)\big)$.  
 We have $f^l(1^-) <  f^l(0)$ when $l$ is not divisible by $q$ and $f^l(1^-) > f^l(0)$ when $l$ is a multiple of $q$. 
 The decomposition into disjoint intervals
\begin{equation}
f^n(I)=  I \setminus\bigcup_{l= 1}^n\big[f^l(1^-), f^l(0)\big), \quad{\rm when } \quad 0 \le n \le q-1,
\label{eq13}
\end{equation}
and
\begin{equation}
f^n(I)=  I \setminus\bigcup_{l= n-q+1}^n\big[f^l(1^-), f^l(0)\big), \quad{\rm when } \quad  n \ge q,
\label{eq14}
\end{equation}
holds true. Moreover $f^n(I)$ has Lebesgue measure $\le (\lambda^q\mu^p)^{\lf n/q\rf}$.

\end{proposition}

  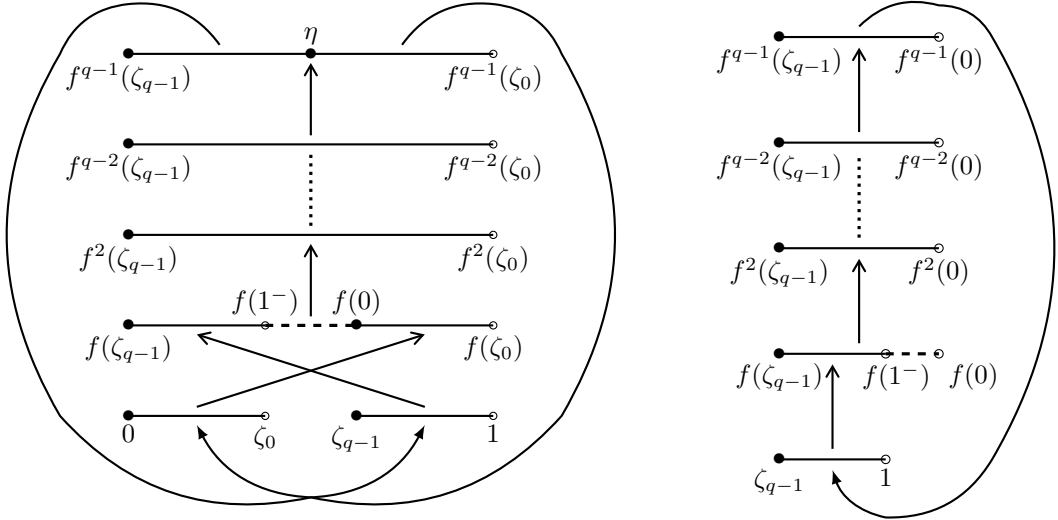
\begin{figure}[ht] 
 
\begin{tikzpicture}[scale = 0.6]
{\footnotesize
\draw[thick]  (0,0) -- (3,0);
\draw[thick]  (5,0) -- (8,0);
\draw[thick]   (0,2) -- (3,2);
\draw[thick]   (4.8,2) -- (8,2);
\draw[dashed][very thick]   (3,2) -- (5,2);
\draw[thick]   (0,4) -- (8,4);

\draw[thick]   (0,6) -- (8,6);
\draw[thick]   (0,8) -- (8,8);

\draw[dotted] [very thick]   (4,4.2) -- (4,5.8);
\draw[thick][->] (1.5,0.2) -- (6.5,1.8);
\draw[thick][->] (6.5,0.2) -- (1.5,1.8);
\draw[thick][->] (4,6.2) -- (4,7.8);
\draw[thick][->] (4,2.2) -- (4,3.8);
\draw  (0,0) node[below ] {$0$} node{$\bullet$};
\draw  (0,2) node[below ] {$f(\zeta_{q-1})$} node{$\bullet$};
\draw  (0,4) node[below ] {$f^2(\zeta_{q-1})$} node{$\bullet$};
\draw  (0,6) node[below ] {$f^{q-2}(\zeta_{q-1})$} node{$\bullet$};
\draw  (0,8) node[below ] {$f^{q-1}(\zeta_{q-1})$} node{$\bullet$};

\draw  (8,0) node[below ] {$1$}  (8,0) circle (0.08) ;
\draw  (8,2) node[below ] {$f(\zeta_{0})$} (8,2) circle (0.08);
\draw  (8,4) node[below ] {$f^2(\zeta_{0})$} (8,4) circle (0.08);
\draw  (8,6) node[below ] {$f^{q-2}(\zeta_{0})$} (8,6) circle (0.08) ;
\draw  (8,8) node[below ] {$f^{q-1}(\zeta_{0})$} (8,8) circle (0.08);
\draw  (4,8) node[above ] {$\eta$} node{$\bullet$};

\draw  (3,0) node[below ] {$\zeta_{0}$} (3,0) circle (0.08);
\draw  (5,0) node[below] {$\zeta_{q-1}$} node{$\bullet$};
\draw  (3,2) node[above  ] {$f(1^-)$} (3,2) circle (0.08) ;
\draw  (5,2) node[above ] {$f(0)$} node{$\bullet$};

\tikzstyle{quadri}=[rectangle,draw,fill=yellow!50,text=blue]
 \tikzstyle{estun}=[->,>=latex,very thick,dotted]

\tikzstyle{estum}=[->,draw, thick,>=latex];
\draw[estum] (2,8.2) to[bend right] (0,9.1) to[bend right] (-1.5,8) to[bend right] (-1.5,0)   to[bend right] (4,-1.8) to[bend right] (6.5,-0.2);

\draw[estum] (6,8.2) to[bend left]    (8,9.1)  to[bend left]    (9.5,8) to[bend left]    (9.5,0)  to[bend left]  (4,-1.8) to[bend left] (1.5,-0.2);
}

\end{tikzpicture}
\qquad
\begin{tikzpicture}[scale = 0.7]
{\footnotesize
\draw[thick]  (0,0) -- (2,0);
\draw[thick]   (0,2) -- (2,2);
\draw[thick]   (0,4) -- (3,4);
\draw[thick]   (0,6) -- (3,6);
\draw[thick]   (0,8) -- (3,8);

\draw[dashed][very thick]   (2,2) -- (3,2);

\draw[dotted] [very thick]   (1.5,4.2) -- (1.5,5.8);

\draw[thick][->] (1,0.2) -- (1,1.8);
\draw[thick][->] (1.5,6.2) -- (1.5,7.8);
\draw[thick][->] (1.5,2.2) -- (1.5,3.8);

\draw  (0,0) node[below ] {$\zeta_{q-1}$} node{$\bullet$};
\draw  (0,2) node[below ] {$f(\zeta_{q-1})$} node{$\bullet$};
\draw  (0,4) node[below ] {$f^2(\zeta_{q-1})$} node{$\bullet$};
\draw  (0,6) node[below ] {$f^{q-2}(\zeta_{q-1})$} node{$\bullet$};
\draw  (0,8) node[below ] {$f^{q-1}(\zeta_{q-1})$} node{$\bullet$};

\draw  (3,2) node[below right] {$f(0)$} (3,2) circle (0.08);
\draw  (3,4) node[below ] {$f^2(0)$} (3,4) circle (0.08);
\draw  (3,6) node[below ] {$f^{q-2}(0)$} (3,6) circle (0.08) ;
\draw  (3,8) node[below ] {$f^{q-1}(0)$} (3,8) circle (0.08);

\draw  (2,0) node[below ] {$1$} (2,0) circle (0.08);
\draw  (2.2,2) node[below ] {$f(1^-)$} (2,2) circle (0.08) ;

\tikzstyle{quadri}=[rectangle,draw,fill=yellow!50,text=blue]
 \tikzstyle{estun}=[->,>=latex,very thick,dotted]

\tikzstyle{estum}=[->,draw, thick,>=latex];

\draw[estum] (1.5,8.2) to[bend left] (3,8.6) to[bend left] (4,8) to[bend left] (4,0)  to[bend left] (2,-1.1) to[bend left] (1,-0.2);

}

\end{tikzpicture}

\caption{Dynamics of the map $f$ with $\zeta_0>0$ on the left and $\zeta_0= 0$ on the right. The arrows indicate the action of $f$ on the intervals.}
\end{figure}

\begin{proof}
Let us consider the partition of $I$
$$
I = [\zeta_{0},\zeta_{1})\cup [\zeta_{1},\zeta_{2})\ldots \cup [\zeta_{q-2},\zeta_{q-1})\cup [\zeta_{q-1}, \zeta_0)
$$
into disjoint circular intervals.  The action of $f$ on these intervals is drawn in Figure 6. 
The inclusions  
$$
f([\zeta_{lp-1}, \zeta_{lp}) \subseteq [\zeta_{(l+1)p-1},\zeta_{(l+1)p})
$$
hold for  $1 \le l \le q$ and we have equality 
$$
f([\zeta_{lp-1}, \zeta_{lp}) = [\zeta_{(l+1)p-1},\zeta_{(l+1)p})
$$
for $1\le l \le q-1$.  However,  for $l=q$, we have a strict inclusion:
$$
 f([\zeta_{q-1}, \zeta_0)) =[ \zeta_{p-1}, f(1^-)) \cup [ f(0), \zeta_p) = [\zeta_{p-1},\zeta_p) \setminus [f(1^-), f(0)).
$$
Thus \eqref{eq13} holds for $n=1$. Applying  $f$ to this decomposition,  we observe the appearance of a second ``hole'' $[f^2(1^-), f^2(0))$ contained in  the  interval $[\zeta_{2p-1},\zeta_{2p})$, just over the first hole $[f(1^-), f(0))$ dotted on Figure 6. Iterating $f$ again, we obtain  \eqref{eq13} for $1 \le n \le  q$ (note that the formulae \eqref{eq13} and \eqref{eq14} coincide for $n=q$). At the  $q$-th iteration, each interval 
$\big[\zeta_{l p-1}, \zeta_{l p}\big), 1 \le  l \le q$, has been holed and $\eta$ does not belong to $f^q(I)$.
Then, $f$ exchanges  the intervals and contracts them. We thus obtain formula \eqref{eq14} for $n \ge q$. 

We finally prove the bound 
$
| f^n(I)| \le (\lambda^q\mu^p)^{\lf n/q\rf}
$ 
for the Lebesgue measure of the iterated images $f^n(I)$. We claim that
\begin{equation}
| f^{n+q}(I) | = \lambda^q \mu^p | f^n(I) | 
\label{eq19}
\end{equation}
for every integer $n\ge 0$. Suppose for instance $n\ge q$ and rewrite \eqref{eq14} in the form
\begin{equation}
 f^{n}(I)= \bigcup_{l=n-q+1}^{n} \big [\zeta_{l p-1},f^{l}(1^-)\big) \cup \big[f^l(0), \zeta_{l p} \big)    .
 \label{eq15}
\end{equation}
Each interval $H=   [\zeta_{l p-1},f^{l}(1^-)$  or $H =[f^l(0), \zeta_{l p} ) $ involved in the disjoint union \eqref{eq15} is contained either in $[0,\eta)$ or in $[\eta,1)$. Keeping track of the iterated images 
$$
f^m(H)=   \big[\zeta_{(l+m) p-1},f^{l+m}(1^-) \big)
\quad{\rm or} \quad
f^m(H) = \big[f^{l+m}(0), \zeta_{(l+m) p}  \big), 
$$
for $1 \le m \le q $ in the decomposition \eqref{eq15}  at  level  $n+m$ (instead of $n$), we observe that  $f^m(H)$ is an interval contained in $[0, \eta)$ for $q-p$ values of $m$ and in $[\eta,1)$ for the $p$ remaining values of $m$ (see Figure 5). 
Notice that the image $f(J)$ of any interval $J$ contained in $[0,\eta)$ (resp. $[\eta,1)$) is an interval whose  lenght equals $\lambda | J|$ (resp. $\lambda\mu | J|$). The length of $f^{m+1}(H)$ equals the length of $f^m(H)$ multiplied either by $ \lambda$ or by $\lambda\mu$,
according whether $f^m(H)$ is contained in $[0,\eta)$ or in $[\eta,1)$. It follows that
$$
| f^{q}(H) |  = \lambda^{q-p} (\lambda\mu)^p | H | =  \lambda^{q} \mu^p | H | .
$$
Summing over the disjoint intervals $H$ occurring in \eqref{eq15}, we obtain \eqref{eq19} for $n\ge q$. The proof for
$n\le q-1$ is similar, now based on the decomposition \eqref{eq13}. Using euclidean division, write
$n= q\lf n/q \rf + r$, with $0 \le r <q$. Equation \eqref{eq19} yields the required bound
$$
 | f^n(I) | = (\lambda^q\mu^p)^{\lf n/q\rf} | f^r(I) | \le  (\lambda^q\mu^p)^{\lf n/q\rf}.
$$

\end{proof}

We deduce from Proposition 6 the following explicit decomposition of the images $f^n(I)$.
 
  \begin{corollary}
   Let $n\ge q$.
 Denote by $\bar{p}$ the multiplicative inverse of $p$ modulo $q$. For every integer $l$ with $ n-q+1\le l \le n$, let  $m $ be the unique integer in the interval $ n-q+1\le m \le n$ which is congruent to $l + \bar{p}$ modulo $q$. Then, the decomposition into disjoint intervals
 $$
 f^{n}(I)= \bigcup_{l=n-q+1}^{n}  [f^l(0),f^{m}(1^-)).
 $$
 holds true. For every integer $l$ with $n-q+1 \le l \le n$,  the interval $[f^l(0),f^{m}(1^-))$ contains the point $\zeta_{lp}$. 
 \end{corollary} 
 
\begin{proof} Recall the decomposition \eqref{eq15} and
observe that $\zeta_{m p-1} = \zeta_{l p}$ and that 
 $ l   \mapsto   m $ is a bijection  of the set $ \{n-q+1, \ldots , n\} $. Collecting the intervals involved in  \eqref{eq15} by pairs $(l,m)$, we find 
 $$
  f^{n}(I)= \bigcup_{l=n-q+1}^{n} \big [\zeta_{l q},f^{m}(1^-)\big) \cup \big[f^l(0), \zeta_{l q} \big)   = \bigcup_{l=n-q+1}^{n}  [f^l(0),f^{m}(1^-)).
 $$
\end{proof}

 It follows from the corollary that
 \begin{equation}
 C := \bigcap_{n\ge0} f^{n}(I)= \{\zeta_0, \dots, \zeta_{q-1} \} = \phi(I).
 \label{eq16}
 \end{equation}
 Indeed, $ \{\zeta_0, \dots, \zeta_{q-1} \}$ is contained in $f^n(I)$, for every $n \ge 1$. 
 The image $f^n(I)$, for $n\ge q$, equals the union of $q$ intervals whose lengths  shrink to $0$, as $n \to \infty$, noting that Proposition 6 delivers the bound
 $$
 | f^n(I) | \le (\lambda \mu^{p/q})^{q\lf n/q \rf} \le \begin{cases}
\max(\lambda,\lambda\mu)^{q\lf n/q\rf} & {\rm if} \quad \lambda\mu <1
 \\
 (\mu^{p/q-r_{\lambda,\mu}})^{q\lf n/q \rf} & {\rm  if} \quad  \lambda\mu \ge 1,
 \end{cases}
$$
where in both upper bounds, the numbers $\max(\lambda,\lambda\mu)$ and $ \mu^{p/q-r_{\lambda,\mu}}$ are less than $1$. 
 This proves \eqref{eq16}. Since these $q$ disjoint  intervals rotates under the action of $f$, we obtain as well that the the $\omega$-limit  set $\omega(x)$  equals  $C$,  for any $x\in I$. The proof of the assertion (ii) of Theorem 3 is complete.

\section{The right end point}   
 
 We deal  here with the exceptional value $\delta = \delta(\lambda,\mu,p/q)$. Let us first explain the reasons why the arguments expanded in Section 7 do not apply to this value. Indeed, \eqref{eq8}
 shows that $\phi(0^-)= 0$ in this case. Then
 $$
 \zeta_{q-1} = \phi(1^-) = \phi(0^-) + 1 = 1,
 $$
and the set $\{\zeta_0, \dots, \zeta_{q-1}\} $ cannot be of course  an $f$-cycle, since it contains the point $1$ which is outside the set of definition $I$ of the map $f$. We slightly modify  the map $f$ in order that the obstruction no longer holds. 

Put  $J=(0,1]$. As usual, let $\lambda, \mu, \delta$ be three real numbers with
 $$
 0< \lambda  <1,  \,  \mu >0,  \, 1-\lambda < \delta <d_{\lambda,\mu}
 $$
and let $\rho = \rho_{\lambda,\mu, \delta}$ be the rotation number of $f=f_{\lambda,\mu,\delta}$. Recall the associated lift $F= F_{\lambda, \mu ,\delta}$ and the conjugation $\phi = \phi_{\lambda,\mu,\delta,\rho}$. We introduce  three functions
$f^-, F^-$ and $\phi^-$ which are  the left limit of $f, F$ and $\phi$ respectively.

\begin{definition} 
 (i)
  Let    $f^-: J \mapsto J$ be the map defined by 
  $$
f^-(x) = f(x^-) =  \begin{cases} \lambda x + \delta,    \quad  if \quad 0< x \le \eta , 
\\
\mu(\lambda x + \delta -1),     \quad   if  \quad \eta < x \le 1.
\end{cases}
$$
\\
(ii). Let    $F^-: \R \mapsto \R$ be the map defined by 
$$
F^-(x) = F(x^-) = \begin{cases} F(x)     \quad  if \quad x\in \R \setminus \Z , 
\\
F(x) -\delta + \mu(\lambda+ \delta -1)     \quad   if \quad  x \in \Z , 
\end{cases}
$$
  \\
(iii)  Let    $\phi^-: \R \mapsto \R$ be the map defined by 
$$
\begin{aligned}
\phi^-(y) = &  \, \phi(y^-) 
\\
= &
\lc y\rc   - {\lambda + \delta -1 \over \lambda} + \sum_{k\ge 0}\lambda^k\mu^{\lc  y\rc   -\lc y-k\rho\rc  }\left( {\lambda +\delta -1 \over \lambda} + \lc y-(k+1)\rho \rc   - \lc y-k\rho\rc  \right).
\end{aligned}
$$
\end{definition}

The maps  $f$ and $f^-$ share almost the same dynamical behaviour and we  present the analogies,  omitting the proofs which follow the lines of Sections 3 and 4.
The two maps $f$ and $f^-$ coincide on $(0,1)\setminus \{\eta\}$ and differ  at the critical point $\eta = (1-\delta)/\lambda$ where $f(\eta)= 0$ and $f^-(\eta)= 1$. 
Thus,  any $f$-orbit contained in $(0,1)$ does not contain the point $\eta$ and is also an $f^-$-orbit. 
The  function $F^-$ turns out to be a lift for the circle map $f^-$,  identifying now the circle $\R/\Z$ with the interval $J$. It follows that both maps $f$ and $f^-$ have the same rotation number $\rho= \rho_{\lambda,\mu,\delta}$. One can show that when $\rho$ is irrational, or when $\rho = p/q$ is rational and
\begin{equation}
\qquad  \delta\left(\lambda,\mu,(p/ q)^-\right) < \delta \le \delta\left(\lambda,\mu , p/q\right), 
 \label{eq18}
\end{equation}
the functional equation
$
F^-(\phi^-(y) ) = \phi^-(y+ \rho)
$
holds for any $y\in \R$. Moreover, when $\rho$ is irrational, the closure
$
 \overline{ \phi^-(J)} = \overline{\phi(I)} 
$
is the  Cantor set  considered in Theorem 3 (i) and any $f^-$-orbit approaches of this Cantor set. When \eqref{eq18} holds, every $f^-$-orbit approaches cyclically of  the periodic $f^-$-cycle
$$
\phi^-(J) = \{ \phi^-(1/q), \dots , \phi^-(1)\} = \{ \zeta_0, \dots, \zeta_{q-1}\} = \phi(I)
$$
of order $q$.

From now, let us fix  $\delta = \delta(\lambda,\mu,p/q)$ and put $f=f_{\lambda,\mu,\delta}$.  Since $\delta$ belongs to the interval  \eqref{eq18},  $\phi(I)$ is an $f^-$-cycle of order $q$ containing  the point $1$.
Let $\O$ be an $f$-orbit. If $0 \notin \O$, then $\O$ is also an $f^-$-orbit which 
 converges to $\phi(I)$.  It follows  in particular that there  exists no finite $f$-cycle. Indeed, arguing as in Section 7, this finite cycle should be an attractor for any $f$-orbit $\O$, and it would be equal to $\phi(I)$,  which is impossible since $1\in \phi(I)$. Suppose now that $0\in \O$. Then $0$ appears only once in $\O$. If not, $\O$ would contain a finite $f$-cycle. Hence, some tail of $\O$ does not contain $0$ and  $\O$ converges as well to $\phi(I)$. Part (iii) of Theorem 3 is established.

 \noindent
{\bf Acknowledgements.} 
We graciously acknowledge the support of R\'egion Provence-Alpes-C\^ote d'Azur through the project APEX {\it Syst\`emes dynamiques: Probabilit\'es et Approximation Diophantienne} PAD, CEFIPRA through the project No. 5801-B and the program  MATHAMSUD projet No. 38889TM DCS: Dynamics of Cantor systems.

\vskip 5mm

\centerline{\scshape {\rm Michel} Laurent {\rm and } {\rm Arnaldo} Nogueira }

{\footnotesize
 \centerline{Aix Marseille Univ, CNRS, Centrale Marseille, I2M, Marseille, France}
 \centerline{michel-julien.laurent@univ-amu.fr and  arnaldo.nogueira@univ-amu.fr}}


\vskip 5mm

\begin{thebibliography}{100}

\bibitem{AdDa} 
W. W Adams and J.L. Davison. 
{\it A remarkable class of continued fractions}, Proc. Amer. Math. Soc.  {\bf 65} (1977), 194-198.




\bibitem{Boh}
P.E. B$\ddot{\o}$hmer.
{\it  $\ddot{U}ber$ die Transzendenz gewisser dyadischer Br$\ddot{u}$che}, Math. Ann. {\bf 96} (1927), 367-377.

\bibitem{Bo} M. Boshernitzan. {\it Dense orbits of rationals}, Proceedings of the American Mathematical Society Vol. 117, Number 4 (1993), 1201-1203.

\bibitem{BoSa} 
J. Bowman and S. Sanderson. 
{\it Angel's staircases, sturmian sequences, and trajectories on homothetic surfaces}, Arxiv: 1806.04129v2 [math.GT].

\bibitem{BoBo} 
J.M. Borwein and P.B Borwein. 
{\it On the generating function of the integer part: $[ n \alpha + \gamma]$}, J. Number Theory {\bf 43} (1993), 293-318.



\bibitem{Br}
J. Br\'emont.
{\it Dynamics of injective quasi-contractions}, Ergodic. Th. and Dyn. Syst. {\bf 26} (2006), 19-44.

\bibitem{Bu}
Y. Bugeaud.
{\it Dynamique de certaines applications contractantes, lin\'eaires par morceaux, sur [0,1[},
C. R. Acad. Sci. de Paris {\bf 317} S\'erie I  (1993), 575-578.

\bibitem{BuC}
Y. Bugeaud and J.-P. Conze.
{\it Calcul de la dynamique d'une classe de transformations  lin\'eaires contractantes mod 1 et arbre de Farey},
Acta Arithmetica {\bf LXXXVIII.3}  (1999), 201-218.

\bibitem{Co}
R. Coutinho.
{\it Din\^amica simb\'olica linear}, Ph. D. Thesis, Instituto Superior T\'ecnico, Universidade T\'ecnica de Lisboa, February 1999. 

\bibitem{Da}
L.V. Danilov.
{\it Some class of transcendental numbers}, Math. Zametki {\bf 12} (1972), 149-154; Math. Notes {\bf 12} (1972), 524-527.

\bibitem{DH}
E. J. Ding and P. C. Hemmer. {\it Exact treatment of mode locking for a piecewise linear map}, Journal of Statistical Physics, {\bf 46} (1987), 
99-110.


\bibitem{FC}
O. Feely and L. O. Chua.  {\it The effect of Integrator Leak in $\Sigma-\Delta$ Modulation}, 
IEEE Transactions on Circuits and Systems,  {\bf 38} (1991), 1293-1305.


\bibitem{Ha}
M. Hata. {\it Neurons. A mathematical ignition}, Series on Number Theory and its Applications, Vol. 9 (2015),  World Scientific Publishing.





\bibitem{JaOb}
S. Janson and C. $\ddot{\rm O}$berg. {\it A piecewise contractive dynamical system and election methods}, Arxiv: 1709.06398v1 [Math. DS].

\bibitem{Ko} 
T. Komatsu. {\it A certain power series and the inhomogeneous  continued fraction expansions}, J. Number Theory {\bf 59} (1996), 291-312.



\bibitem{LN} M. Laurent and A. Nogueira. {\it Rotation number of contracted rotations}, Journal of Modern Dynamics, Volume 12 (2018), 175-191.

\bibitem{LoVdPA}
J.H. Loxton  and A.J. van der Poorten.
{\it Arithmetic properties of certain functions in several variables III},
Bull. Austral. Math. Soc., {\bf 16}  (1977), 15-47.

\bibitem{LoVdPB}
J.H. Loxton  and A.J. van der Poorten.
{\it Transcendence and algebraic independence by a method of Mahler}, in 
{\it Transcendence Theory: Advances and applications}, ed. by A. Baker and D.W.  Masser, Academic Press (1977), 211-226.


\bibitem{NS}
J. Nagumo and S. Sato.  {\it  On a response characteristic of a mathematical neuron model}, Kybernetik {\bf 10(3)} (1972), 155-164.



\bibitem{Ni}
K. Nishioka. {\it  Mahler Functions and Transcendence}, Springer Lecture Notes in Mathematics,  Vol. 1631 (1996).

\bibitem{NST} 
K. Nishioka, I. Shiokawa and J. Tamura. 
{\it Arithmetical properties of certain power series}, J. Number Theory {\bf 42} (1992), 61-87.


\bibitem{NP}
A. Nogueira and B. Pires. {\it Dynamics of piecewise contractions of the interval}, Ergodic Theory and Dynamical Systems,
Volume 35, Issue 7 (2015), 2198-2215.

\bibitem{NPR}
A. Nogueira, B. Pires and R. Rosales. {\it Topological dynamics of piecewise $\lambda$-affine maps}, Ergodic Theory and Dynamical Systems,
Volume 38 (2018), 1876-1893.


\bibitem{RT}
F. Rhodes and C. Thompson. {\it Rotation numbers for monotone functions on the circle}, J. London Math. Soc. (2) {\bf 34} (1986), 360-368. 

\bibitem{RT2}
F. Rhodes and C. Thompson. {\it Topologies and rotation numbers for families of monotone functions on the circle}, J. London Math. Soc. (2) {\bf 43} (1991), 156-170. 


\end{thebibliography}
\end{document}